\documentclass[12pt,a4paper]{article}

\usepackage[utf8]{inputenc}
\usepackage{csquotes}
\usepackage[english]{babel}
\usepackage[T1]{fontenc}
\usepackage{tgpagella}
\usepackage[backend=biber]{biblatex}
\usepackage{amsmath}
\usepackage{amsfonts}
\usepackage{amssymb}
\usepackage{amsthm}
\usepackage{graphicx}
\usepackage{blkarray}
\usepackage[left=2cm,right=2cm,top=2cm,bottom=2cm]{geometry}

\usepackage{hyperref}
\usepackage{url}
\usepackage{authblk}                 % used for author declaration with affiliations
\usepackage{caption}
\usepackage{subcaption}

\usepackage{tikz}
\usetikzlibrary{shapes.geometric}
\usetikzlibrary{matrix}

\DeclareMathOperator{\Z}{\mathbb{Z}}

\DeclareMathOperator{\R}{\mathbb{R}}
\DeclareMathOperator{\C}{\mathbb{C}}
\DeclareMathOperator{\F}{\mathbb{F}}

\DeclareMathOperator{\Mat}{\mathrm{Mat}}
\DeclareMathOperator{\diag}{\mathrm{diag}}

\DeclareMathOperator{\BH}{\mathrm{BH}}
\DeclareMathOperator{\GW}{\mathrm{GW}}

\DeclareMathOperator{\re}{\mathrm{Re}}

\newcommand{\genlegendre}[4]{%
  \genfrac{(}{)}{}{#1}{#3}{#4}%
  \if\relax\detokenize{#2}\relax\else_{\!#2}\fi
}
\newcommand{\legendre}[3][]{\genlegendre{}{#1}{#2}{#3}}

\newtheorem{theorem}{Theorem}
\newtheorem{proposition}{Proposition}
\newtheorem{lemma}{Lemma}
\newtheorem{corollary}{Corollary}
\theoremstyle{definition}
\newtheorem{example}{Example}
\theoremstyle{definition}
\newtheorem{definition}{Definition}
\theoremstyle{definition}
\newtheorem{remark}{Remark}

\newcommand{\red}[1]{\textcolor{red}{#1}}

\title{Maximal determinants of matrices over the roots of unity}
\author[1, 2, *]{Guillermo Nuñez Ponasso}
\affil[1]{\small Graduate School of Information Sciences: Division of Mathematics, Tohoku University, Sendai, Miyagi, Japan}
\affil[2]{\small Dept. of Electrical \& Computer Engineering, Worcester Polytechnic Institute, Worcester, MA, USA}
\affil[*]{correspondence: \href{mailto:guillermo.carlo.nunez.a7@tohoku.ac.jp}{guillermo.carlo.nunez.a7@tohoku.ac.jp}\\\href{mailto:gcnunez@wpi.edu}{gcnunez@wpi.edu}}

\begin{document}

%% === Title == %%
\maketitle

\abstract{We study the maximum absolute value of the determinant of matrices with entries in the set of $\ell$-th roots of unity~\textemdash~this is a generalization of $D$-optimal designs and Hadamard's maximal determinant problem, which involves $\pm 1$ matrices. For general values of $\ell$, we give sharpened determinantal upper bounds and constructions of matrices of large determinant. The maximal determinant problem in the cases $\ell=3$, $\ell=4$ is similar to the classical Hadamard maximal determinant problem for matrices with entries $\pm 1$, and many techniques can be generalized. For $\ell=3$ we give and additional construction of matrices with large determinant, and calculate the value of the maximal determinant over $\mu_3$ for all orders $n<14$. Additionally, we survey the case $\ell=4$ and exhibit an infinite family of maximal determinant matrices over the fourth roots of unity.}
%%%%%%%%%%%%%%%%%%%%%%%%%%%%%%
  %% === INTRODUCTION === %%
%%%%%%%%%%%%%%%%%%%%%%%%%%%%%%

\section{Introduction}
Maximal determinant $\pm 1$ matrices, or $D$-optimal designs,  are interesting objects in statistics since they can be used to construct designs of experiments which have many nice statistical properties \cite{Box-Draper-71, Aguiar-DOptimalTutorial, Mitchell-DOptimal, Smith-DCriterion}. A particularly important family of maximal determinant $\pm 1$ matrices are \textit{Hadamard matrices} \cite{Horadam-HadamardBook}, which have additional applications to coding theory \cite{Assmus-Key, Seberry-Applications}, philogenetics \cite{Gonzalez-Advances, Hendy-Phyl, Hendy-Trees, Penttila-Conjugation}, and compressed sensing \cite{OCathain-Compressed, Vaz-Compressed}, among others.\\
Hadamard matrices are $\pm 1$ matrices $H$ of order $n$ satisfying $HH^{\intercal}=nI_n$, where $I_n$ is the identity matrix of size $n$. A natural generalization is that of \textit{complex Hadamard matrices}: matrices $H$ of order $n$ with entries taken from the complex unit circle satisfying $HH^*=nI_n$. These matrices have received attention \cite{Szollosi-Thesis, Tadej2006} due to their applications in the field of quantum information theory and quantum computation \cite{Bengtsson-GQS, Werner-Teleportation,Zauner}. An important class of complex Hadamard matrices is the class of \textit{Butson-type Hadamard matrices}: these are complex Hadamard matrices whose entries are taken from the set of $\ell$-th roots of unity. The set of Butson-type matrices of order $n$ over the $\ell$-th roots is typically denoted by $\BH(n,\ell)$. These matrices have interesting connections to real Hadamard matrices \cite{CCDL, OCathain-Morphisms, Turyn}, and generalized Hadamard matrices \cite{deLauney-SurveyGHM}.\\
Historically, the study of complex Hadamard matrices was initiated by J.J. Sylvester prior to the study of real Hadamard matrices \cite{Sylvester-InverseOrthogonal}~\textemdash in fact, he considered an even more general object now known as \textit{type-II} matrix \cite{Chan-TypeII}. Meanwhile, the study of real Hadamard matrices began alongside the study of maximal determinant $\pm 1$ matrices in Hadamard's seminal paper \cite{Hadamard-Determinants}. Hadamard showed the following:
\begin{theorem}[Hadamard's determinantal bound \cite{Hadamard-Determinants}]\label{thm-Hadamard} If $M$ is a matrix of order $n$ with entries belonging to the complex unit disk $D=\{z\in\C:|z|\leq 1\}$, then
  \begin{equation}
    |\det M |\leq n^{n/2}.
  \end{equation}
  Furthermore, equality holds above if and only if $MM^*=nI_n$.
\end{theorem}
Complex Hadamard matrices exists at all possible orders $n$ since the Fourier matrix $F_n$~\textemdash~the character table of the cyclic group of order $n$~\textemdash~ is a complex Hadamard matrix, and in particular a $\BH(n,n)$; more generally, the character table of an abelian group of order $n$ and exponent $\ell$ give examples of $\BH(n,\ell)$ matrices. However, restricting the set of entries we may yield non-existence: for example, in the $\pm 1$ case, Hadamard matrices can only exists at orders $n=1,2$ or $n$ a multiple of $4$. The existence problem then becomes very challenging: it has been conjectured almost a hundred years ago \cite{Paley} that real Hadamard matrices exist for all orders $n=4m$, yet there is still no proof or refutation for this claim.\footnote{Currently, the smallest open case for the existence of a real Hadamard matrix is $n=668$.}\\

Throughout the paper, we will denote the set of $\ell$-th roots of unity as
\begin{equation}
  \mu_{\ell} = \{1,\zeta_{\ell},\zeta_{\ell}^{2},\dots,\zeta_{\ell}^{\ell-1}\},
\end{equation}
where $\zeta_{\ell}$ is a primitive $\ell$-th root of unity. In the particular cases $\ell=3$ and $4$, we denote the corresponding primitive roots as $\omega:=\zeta_3$, and $i:=\zeta_4$. From a matrix $A$ with entries in the set $\{0,1,\dots,\ell-1\}$ we can obtain a unique matrix $M$ over the $\mu_{\ell}$ by letting $M_{ij}=\zeta_{\ell}^{A_{ij}}$, and conversely from a matrix $M$ over $\mu_{\ell}$ we can obtain a unique matrix $A$ with entries in $\{0,1,\dots,n-1\}$. We call the matrix $A$ the \textit{logarithmic form} of $M$.

\begin{definition}\label{def-MonomialEquivalence}
Two matrices $M$ and $N$ with entries in $\mu_{\ell}$ are \textit{equivalent}, or \textit{monomially equivalent}, if and only if there exist a pair of permutation matrices $(P,Q)$ and a pair of diagonal matrices $(\Delta_1,\Delta_2)$ with diagonal entries in $\mu_{\ell}$ such that:
\begin{equation}
  N = (\Delta_1 P)M(\Delta_2 Q)^*.
\end{equation}
\end{definition}

Hadamard's maximal determinant problem is simple to state: If $n$ is an integer, find the maximum value of the determinant for a matrix of order $n$ with entries in the set $\mu_2=\{+1,-1\}$. While Hadamard's maximal determinant problem has received much attention (see \cite{OCathain-SurveyMaxDet} for a survey of results), generalizations of this problem to larger sets of entries are still largely unexplored. It appears that the only generalization that was previously considered is due to J.H.E. Cohn \cite{Cohn-ComplexDOptimal} where he studied complex $D$-optimal designs; i.e., maximal determinant matrices over the set of fourth roots of unity $\mu_4=\{1,i,-1,-i\}$.\\

In this paper, we extend Hadamard's maximal determinant problem to matrices over the roots of unity: Letting $\mathcal{M}_{\ell}(n)=\{M\in \Mat_{n}(\C):M_{ij}\in\mu_{\ell},\text{ for } 1\leq i,j\leq n\}$, determine the value of
\begin{equation}
  \gamma_{\ell}(n)=\max\{|\det M|:M\in\mathcal{M}_{\ell}(n)\}.
\end{equation}
This problem depends heavily on the behavior of the sums of $\ell$-th roots of unity, particularly on their minimal absolute value. In the cases $\ell=3,4,$ and $6$, the ring $\Z[\zeta_{\ell}]$ (where $\zeta_{\ell}$ is a primitive $\ell$-th root of unity) has no accumulation points in the induced Euclidean topology and this simplifies the analysis of the maximal determinant problem.\\
In the case $\ell=4$ there is much we can say from existing results, the reason for this is that $\mu_2\subset\mu_4$, so many examples of maximal determinant matrices over $\mu_2$ are also maximal determinant over $\mu_4$. Additionally, the Turyn morphism \cite{OCathain-Morphisms, Turyn} provides a connection between matrices over $\mu_4$ and matrices over $\mu_2$ of twice the order, which can then be used to ``lift'' maximal determinant matrices from $\mu_2$ to $\mu_4$.\\
The case $m=3$ is more challenging, since $\mu_2\not\subset \mu_3$ and there is no analogue of the Turyn morphism. In addition, this case was previously unexplored so we put more emphasis on it. In the case $m=6$, there is evidence to believe that $\BH(n,6)$ matrices exist at all orders $n$ which are not ruled out by the non-existence conditions of \cite{Winterhof-NonexistenceButson}: namely, that if $n$ has an odd prime factor $p$ with odd multiplicity satisfying $p\equiv 5\pmod{6}$, then there is no $\BH(n,6)$. In the limit, this condition rules out a fraction of $\approx 0.3671$ of all $n$, so the set of unfeasible orders is relatively sparse. Also, many techniques we present rely on the assumption that the sum of $n$ roots of unity in $\mu_{\ell}$ have absolute value $\geq 1$. Cancellation of sums of sixth roots happens at all orders $n>1$, so it seems that a different approach is required

\begin{itemize}

\item[\textendash]In Section \ref{sec-ULBounds} we give upper and lower bounds for the absolute value of the determinant of a matrix over $\mu_{\ell}$, where $\ell$ is arbitrary. In particular, we show a generalization of Barba's bound \cite{Barba-DetBound}, and  Cohn's bound \cite{Cohn-ComplexDOptimal}. We show that Barba matrices, i.e. those matrices satisfying the Barba bound with equality, are equivalent to normal Barba matrices with constant row sum.  We give a construction for large determinant matrices over $\ell$ at orders $n\equiv 1\pmod{\ell}$ which achieves at least 60\% of Hadamard's bound and 70\% of our generalized Barba bound.
\item[\textendash] In Section \ref{sec-3Roots}, we focus on the study of the case $\ell=3$. We give an additional construction of large-determinant matrices at orders $n\equiv 2\pmod{3}$ using generalized Paley cores and the theory of cyclotomy \cite{Storer-CyclotomyBook}. Similar to the maximal determinant problem over $\mu_2$, the problem over $\mu_3$ splits into congruence classes modulo $3$. Therefore, in combination with the general construction at orders $n\equiv 1\pmod{3}$, we have large lower bounds for the determinant at infinitely many orders in each congruence class modulo $3$.
\item[\textendash] In Section \ref{sec-3RootsSmall} continue our study of the case $\ell=3$ and consider maximal determinant matrices over $\mu_3$ of small size. We classify Barba matrices with two different entries, and Barba matrices supported by a strongly regular graph. We develop computational strategies to certify the maximality of candidate matrices with large determinant, extending the methods of \cite{MK-21, Moyssiadis-Kounias, Orrick-15}. With this, we compute the value of the maximal determinant over $\mu_3$ for all orders $n<14$.
\item[\textendash] In Section \ref{sec-4Roots} we survey the case $\ell=4$, giving a new construction of an infinite family of Barba matrices by lifting a construction of maximal determinant matrices in $\mu_2$ at orders $n\equiv 2\pmod{4}$. Additionally, we report a matrix over $\mu_4$ of order $11$ with larger determinant than the previous record in \cite{Cohn-ComplexDOptimal}.
\end{itemize}

%%%%%%%%%%%%%%%%%%%%%%%%%%%%%%%%%%%%%%%%%%%%%%%%
  %%% === SECTION 2 - GENERAL BOUNDS === %%%
%%%%%%%%%%%%%%%%%%%%%%%%%%%%%%%%%%%%%%%%%%%%%%%%
  \section{General upper and lower bounds} \label{sec-ULBounds}
    For matrices with entries in $\mu_2=\{\pm 1\}$, Guido Barba established the following strengthening of Hadamard's bound:

    \begin{theorem}[Barba, 1931 \cite{Barba-DetBound}]\label{thm-BarbaBound}
        Let $M$ be a $n\times n$ matrix with entries $\pm 1$, with $n$ odd. Then
        \begin{equation}
            |\det M|\leq \sqrt{2n-1}(n-1)^{(n-1)/2}.
        \end{equation}
        Furthermore, $M$ meets the bound with equality if and only if $M$ is monomially equivalent to a matrix $B$ satisfying
        \begin{equation}
            BB^* = (n-1)I_n+J_n.
        \end{equation}
    \end{theorem}

    In \cite{Cohn-ComplexDOptimal}, Cohn showed that this result also holds for matrices with entries in $\mu_4$. Later we will show that Theorem \ref{thm-BarbaBound} is also true for matrices with entries in $\mu_3$. For other values of $\ell$, we show a more general form of Theorem \ref{thm-BarbaBound} which depends on the minimal modulus of sums of $\ell$-th roots ``of length $n$''.

  %%%%%%%%%%%%%%%%%%%%%%%%%%%
  % GENERALIZED BARBA BOUND %
  %%%%%%%%%%%%%%%%%%%%%%%%%%%
  \subsection{A generalization of Barba's bound}
  We begin with two auxiliary results from which Barba's determinant bound \cite{Barba-DetBound}, and Cohn's extension \cite{Cohn-ComplexDOptimal}, can be obtained. Both results appeared in Wojtas' paper \cite{Wojtas-Determinants} for symmetric positive-definite matrices. Their proofs can be adapted to the show the equivalent results for Hermitian positive-definite matrices. Recall the following result:

  \begin{theorem}[Muir-Kelvin bound; cf., Theorem 7.8.1. \cite{Horn-Johnson}] \label{thm:Muir-Kelvin} Let $G$ be an $n\times n$ positive-definite matrix. Then
    \begin{equation}
    |\det G|\leq\prod_{i=1}^n g_{ii},
  \end{equation}
  and equality holds above if and only if $G$ is diagonal.
\end{theorem}

    \begin{lemma} \label{lemma-BLemma} Let $B$ be a $k\times k$ Hermitian positive-definite matrix, written in block form as:
    \begin{equation}
        B = \left[\begin{array}{c|c}
         G_{k-1} & \beta\\
         \hline
         \beta^* & b
        \end{array}\right],
    \end{equation}
    where the diagonal entries of $G_{k-1}=(g_{ij})_{ij}$ are all equal to $m$. If  $0<b\leq|\beta_i|$ for all $1\leq i\leq k$, then
    \begin{equation}
    \det B \leq b(m-b)^k,
    \end{equation}
    and $B$ achieves this bound with equality of and only if $|\beta_i|=b$, and $g_{ij}=\beta_i\beta_j^*/b$ for all  $i\neq j$.
    \end{lemma}

    \begin{proof}
        A series of (Hermitian) elementary row and column operations shows that
        \begin{equation}
            \det B=\det \left[\begin{array}{c|c}
                G_{k-1}-\beta\beta^*/b & 0 \\
                \hline
                0 & b
            \end{array}\right]=b\det(G_{k-1}-\beta\beta^*/b).
        \end{equation}
        Sylvester's criterion implies that the Hermitian matrix $D=(G_{k-1}-\beta\beta^*/b)$ is positive-definite, so we can apply the Muir-Kelvin bound (Theorem \ref{thm:Muir-Kelvin}) to it, and obtain:
        \begin{equation}
            \det(G_{k-1}-\beta\beta^*/b)\leq \prod_{i=1}^k\left(m-\frac{|\beta_i|^2}{b}\right)\leq \prod_{i=1}^k (m-|\beta_i|)\leq (m-b)^k.
        \end{equation}
        The first inequality above is an equality if and only if $D$ is diagonal, which is equivalent to $g_{ij}=\beta_i\beta_j^*/b$ for all $i\neq j$. The remaining inequalities are equalities if and only if $|\beta_i|=b$ for all $i$.\qedhere
    \end{proof}

  \begin{proposition}\label{prop-GeneralBarba} Let $G$ be an $n\times n$ Hermitian positive-definite matrix, with diagonal entries equal to $m$. If $b$ is a positive real number such that $b\leq |G_{ij}|$ for all $i\neq j$, then
  \[\det G\leq (m+(n-1)b)(m-b)^{n-1}.\]
  Additionally, $G$ meets this bound with equality if and only if there is a (necessarily unique) diagonal matrix $\Delta=\diag(d_1,\dots,d_n)$, with $|d_i|=1$ for all $i$, such that
  \begin{equation}
  \Delta^*G\Delta=(n-b)I_n+bJ_n.
  \end{equation}
  \end{proposition}
    \begin{proof}
         Proceeding by induction on $n$, we first show that the result holds for $n=2$:
        \begin{equation}
            \det\begin{bmatrix}
                m & g_{12}\\
                g_{12}^* & m
            \end{bmatrix} = m^2-|g_{12}|^2\leq m^2-b^2=(m+b)(m-b).
        \end{equation}
        Now, suppose the result holds for all matrices of order $n-1$ and let $G$ be an $n\times n$ matrix satisfying the hypotheses. Write, 
        \begin{equation}
            G = \left[\begin{array}{c|c}
                 G_{n-1}& \gamma \\
                 \hline
                 \gamma^*& m 
            \end{array}\right],
        \end{equation}
        where $\gamma=(g_{1n},g_{2n},\dots,g_{n-1,n})^{\intercal}$. By linearity of the determinant on rows:

        \begin{equation}\label{eq-ABDecomposition}
            \det G=\det\underbrace{\left[\begin{array}{c|c}
                 G_{n-1}& \gamma \\
                 \hline
                 \mathbf{0}_{n-1}^{\intercal}& m-b 
            \end{array}\right]}_{\textstyle A}+
            \det\underbrace{\left[\begin{array}{c|c}
                 G_{n-1}& \gamma \\
                 \hline
                 \gamma^* & b 
            \end{array}\right]}_{\textstyle B}
        \end{equation}
        If $\det(B)>0$, then by Sylvester's criterion, the matrix $B$ is positive-definite, so we can apply Lemma \ref{lemma-BLemma} to it, and obtain
        \begin{equation}
            \det G = (m-b)\det(G_{n-1})+\det(B)\leq (m-b)\det(G_{n-1})+b(m-b)^{n-1}.
        \end{equation}
        If $\det(B)\leq 0$, then trivially $\det G \leq (m-b)\det(G_{n-1})\leq (m-b)\det(G_{n-1})+b(m-b)^{n-1}.$ Now, use the induction hypothesis on $G_{n-1}$ to obtain:
        \begin{equation}
            \begin{split}
                \det G&\leq (m-b)(m+(n-2)b)(m-b)^{n-2}+b(m-b)^{n-1}\\
                &=(m+(n-1)b)(m-b)^{n-1}.
            \end{split}
        \end{equation}
        The matrix $G$ meets this bound with equality if and only if $\det(B)=b(m-b)^{n-1}$. By Lemma \ref{lemma-BLemma}, this is characterized by $|\gamma_{i}|=|g_{in}|=b$ for all $i$, and $g_{ij}=\gamma_{i}\gamma_j^*/b=g_{in}g_{jn}^*/b$. Let $\Delta=\diag(g_{1n}/b,\dots,g_{n-1,n}/b,1)$, and let $G'=\Delta^* G\Delta$. Then, $g'_{in}=b$ for $1\leq i\leq n-1$, and since the diagonal entries of $\Delta$ have modulus $1$, it follows that $\det G'=\det G$. Decomposing $G'=A'+B'$ as in Equation \ref{eq-ABDecomposition}, we have $\det B'= b(m-b)^{n-1}$, and by Lemma \ref{lemma-BLemma}  $g'_{ij}=g'_{in}g'_{jn}/b=b.$ This implies that
        \begin{equation}
            G'= (n-b)I_n+bJ_n,
        \end{equation}
        from which we conclude the result.\qedhere
        \end{proof}

  To apply Proposition \ref{prop-GeneralBarba} to matrices with entries in $\mu_{\ell}$ we first introduce some terminology and notation: An \textit{$n$-sum of $\ell$-th roots} is a sum of $n$ elements taken from the set $\mu_{\ell}$. Denote by $\sigma_{\ell}(n)$ the minimum absolute value of an $n$-sum of $\ell$-th roots, i.e.
  \begin{equation}
      \sigma_{\ell}(n)=\min \left\{|a_0+a_1\zeta_{\ell}+\dots+a_{\ell-1}\zeta^{\ell-1}|: a_i\in [n],\ \text{ and } \sum_{i=0}^{\ell-1}a_i=n\right\},
  \end{equation}
  where $[n]:=\{0,1,\dots,n\}$. Equivalently, $\sigma_{\ell}(n)$ is equal to the minimal absolute value of the inner product of two vectors of length $n$ with entries in $\mu_{\ell}$.

  \begin{theorem}[Generalized Barba bound. Theorem 5.2.3 \cite{Ponasso-Thesis}]\label{thm-BarbaRoots}
      Let $M$ be a matrix of order $n$ with entries in $\mu_{\ell}$. Suppose that $\sigma_{\ell}(n)>0$. Then,
      \begin{equation}
      |\det M|\leq \sqrt{(n+(n-1)\sigma_{\ell}(n)}(n-\sigma_{\ell}(n))^{(n-1)/2}.
      \end{equation}
      Furthermore, $M$ achieves the bound with equality if and only if there exists a diagonal matrix $\Delta=\diag(d_1,\dots,d_n)$, with $|d_i|=1$ for all $i$, such that the matrix $B=\Delta^*M$ satisfies
      \begin{equation}
          BB^*=(n-\sigma_{\ell}(n))I_n+\sigma_{\ell}(n)J_n
      \end{equation}
  \end{theorem}

  \begin{proof}
  Apply Proposition \ref{prop-GeneralBarba} to the Gram matrix $G=MM^*$, noticing that $|G_{ij}|\geq \sigma_{\ell}(n)$.\qedhere
  \end{proof}
    
    In general, the matrix $B$ of Theorem \ref{thm-BarbaRoots} may be inequivalent to $M$ in the sense of Definition \ref{def-MonomialEquivalence}, since the entries of $\Delta=\diag(d_1,\dots,d_n)$ need not belong to $\mu_{\ell}$. However, for $\ell=3$ we have:

\begin{corollary}\label{cor-Barba3} Let $M$ be an $n\times n$ matrix with entries in $\mu_3 = \{1,\omega,\omega^2\}$. Then, if $n>2$ and $n\equiv 1,2\pmod{3}$, 
\begin{equation}\label{eq-BarbaBound3}
    |\det M|\leq \sqrt{2n-1}(n-1)^{(n-1)/2}.
\end{equation}
Furthermore, $M$ meets this bound with equality if and only if $M$ is monomially equivalent to a matrix $B$ satisfying
    \begin{equation}
        BB^*=(n-1)I_n+J_n.
    \end{equation}
\end{corollary}

\begin{proof}
  The ring $\Z[\omega]$ is the hexagonal lattice in $\C$, and from this, it is easy to see that $n\equiv 1,2\pmod{3}$ implies $\sigma_3(n)=1$. Using Theorem \ref{thm-BarbaRoots}, we find that Inequality \ref{eq-BarbaBound3} holds in this case. One can also show (see Lemma 5.2.1. in \cite{Ponasso-Thesis}) that if $\sigma$ is an $n$-sum of third-roots with $|\sigma|=1$, then
  \begin{equation}
    \begin{split}
        &\text{if } n\equiv 1\pmod{3},\ \sigma\in\{1,\omega,\omega^2\},\text{ and }\\
        &\text{if } n\equiv 2\pmod{3},\ \sigma\in\{-1,-\omega,-\omega^2\}.
    \end{split}
  \end{equation}
  Suppose that $M$ achieves equality in Inequality \ref{eq-BarbaBound3}. Let $G=MM^*$ and $\Delta=\diag(g_{1n}, \dots, g_{n-1,n},1)$, the the matrix $B=\Delta^*M$ satisfies $BB^*=(n-1)I_n+J_n$. To conclude that $B$ is monomially equivalent to $M$ we only need to show that the diagonal entries of $\Delta$ are third roots of unity: By the proof of Proposition \ref{prop-GeneralBarba}, the identity $g_{ij} = g_{in}g_{jn}^*$ holds for $i\neq j$. Since $n>2$, we have in particular that $g_{12}=g_{13}g_{23}^*$; considering this equation modulo $(1-\omega)$, we find
    \begin{equation}
        n\equiv g_{12}=g_{13}g_{23}^* \equiv n^2\equiv 1\pmod{(1-\omega)}
    \end{equation}
    This implies $n\equiv 1\pmod{3}$, and in particular all elements $g_{in}$ are third roots of unity. So the entries of $B$ also belong to $\mu_3$.
\end{proof}

\begin{definition}
We say that an $n\times n$ matrix $B$ with entries of modulus $1$ is a \textit{Barba matrix} if $BB^*=(n-1)I_n+J_n$.
\end{definition}

\begin{remark} The proof of Corollary \ref{cor-Barba3} shows that Barba matrices over the third roots can only exist at orders $n\equiv 1\pmod{3}$.
  \end{remark}

\begin{theorem}\label{thm-BarbaNormal}
    Let $B$ be a Barba matrix with entries in $\mu_{\ell}$, $\ell=2,3,4$. Then there is a normal Barba matrix with entries in $\mu_{\ell}$ and constant row and column sum.
\end{theorem}
\begin{proof}
    By definition, $BB^*=(n-1)I_n+J_n$, and $|\det(B)|=\sqrt{2n-1}(n-1)^{(n-1)/2}$. Since $|\det(B)|=|\det(B^*)|$, then $B^*$ also meets the Barba bound with equality and there is a diagonal matrix $\Delta$ with diagonal entries in $\mu_{\ell}$ such that
    \begin{equation}
        \Delta^* B^*B\Delta=(n-1)I_n+J_n.
    \end{equation}
    Let $N=B\Delta^*$, then $NN^*=N^*N=(n-1)I_n+J_n$. Furthermore, associativity implies $(NN^*)N=N(N^*N)$, so $N$ commutes with $J_n$, i.e., $NJ_n=J_nN$, so $N$ has a constant row sum and column sum.
\end{proof}
\begin{remark} The assumption that $\ell=2,3,4$ is only used here to guarantee that the matrix $N$ will have its entries in $\mu_{\ell}$. We showed this fact here for $\ell=3$; the case $\ell=2$ is found in \cite{Barba-DetBound, Wojtas-Determinants}, and the case $\ell=4$ in \cite{Cohn-ComplexDOptimal}.
\end{remark}

\begin{corollary}[cf. Neubauer-Radcliffe; Brouwer \cite{Brouwer, Neubauer-Radcliffe}]
  There is a Barba matrix with entries $\pm 1$ of order $(q+1)^2+q^2$ for every odd prime power $q$.
\end{corollary}
\begin{proof} Let $B$ be a $\pm 1$ Barba matrix of order $n$; by Theorem \ref{thm-BarbaNormal}, we may assume that $B$ is normal and has constant row and column sum $\alpha$. Let $N=(J_n-B)/2$ be the $(0,1)$-matrix obtained by replacing the entries equal to $+1$ in $B$ with $0$, and $-1$ with $+1$. Since $B$ has constant row sum $\alpha$, we have that $\alpha J_n = BJ_n$ and $NJ_n = J_nN = k J_n$, where $k:=(n-\alpha)/2\in\Z$. Since $B=J_n-2N$, we find
  \begin{equation}
    (n-1)I_n+J_n=BB^{\intercal}=4NN^{\intercal}+(n-4k)J_n.
  \end{equation}
  From where we conclude,
  \begin{equation}
    NN^{\intercal} = \frac{n-1}{4}I_n + (k-\frac{n-1}{4})J_n.
  \end{equation}
  This implies that $N$ is the incidence matrix of a symmetric $2$-design with parameters $(v,k,\lambda)=(n,k,k-(n-1)/4)$. Re-writing these parameters, we can see that there exists an infinite family of designs satisfying this condition:  Since $k>0$, there is a positive real number $t\in\R$ such that $k=t^2$. We then have that $4\lambda=4k-(n-1)=4t^2-x$, where $x=n-1$. Since $(v,k,\lambda)$ are the parameters of a $2$-design, we have $(v-1)\lambda=k(k-1)$. Therefore,
  \begin{equation}
    4x\lambda = 4k(k-1)=4t^2(t^2-1).
  \end{equation}
  Now, substituting $4\lambda=4t^2-x$ we find
  \begin{equation}
    x^2-4t^2x+4t^2(t^2-1)=0.
  \end{equation}
  The solutions to this quadratic equation on $x$ are $x=2t^2\pm 2t$, which implies that $v=x+1=t^2 + (t\pm 1)^2$. Since $2\lambda = 2k-\frac{x}{2}=t^2\mp t$, we have that $\lambda={t\choose 2}$ or $\lambda={t+1\choose 2}$. Finally, since $t^2=k$ is an integer, and $\lambda=t^2\mp t\in \Z$, it follows that $t\in \Z$. So the parameters $(v,k,\lambda)$ can be rewritten as
  \begin{equation}
    (v,k,\lambda)=(t^2+(t+1)^2, t^2,{t\choose 2}),\text{ or } (t^2+(t-1)^2, t^2, {t+1\choose 2}). 
  \end{equation}
  The existence of symmetric $2$-$(t^2+(t+1)^2,t^2,{t\choose 2})$ designs for $t=q$ an odd prime power was established by Brouwer in \cite{Brouwer}. Therefore, there is a Barba matrix of order $q^2+(q+1)^2$ for each odd prime power $q$.

\end{proof}

%%%%%%%%%%%%%%%%%%%%%%%
% GENERAL LOWER BOUND %
%%%%%%%%%%%%%%%%%%%%%%%
\subsection{A general determinantal lower bound}
\begin{proposition}\normalfont \label{prop-ConstRSBoundComp} Let $H$ be an Hadamard matrix of order $n$ with constant row sum. Let $M$ be the bordered matrix,
\begin{equation}
M=\left[
\begin{array}{cc}
1 & \mathbf{1}_n^{\intercal}\\
\mathbf{1}_n & H
\end{array}
\right].
\end{equation}
Then
\begin{equation}
|\det M|\geq (\sqrt{n}+1)n^{n/2}.
\end{equation}
\end{proposition}
\begin{proof}
Since $H$ has a constant row sum, there exists a complex number $s$ such that $HJ_n=sJ_n$. Using the fact that $H$ is Hadamard, it is easy to show that $H$ has a constant column sum $s$, and that $ss^*=n$. We can compute the Gram matrix of $M$ as follows,
\begin{equation}
MM^*=
\begin{bmatrix}
n+1 & (1+s^*)\mathbf{1}_n^{\intercal}\\
(1+s)\mathbf{1}_n & nI_n+J_n
\end{bmatrix}.
\end{equation}
Let $\alpha=1+s$, a series of elementary row and column operations give the following similarity of matrices,
\begin{equation}
\begin{bmatrix}
n+1 & \alpha^*\mathbf{1}\\
\alpha\mathbf{1} & nI_n+J_n
\end{bmatrix}
=\left[
\begin{array}{cc|c}
n+1 & \alpha^* & \alpha^*\mathbf{1}_{n-1}^{\intercal}\\
\alpha &n+1 & \mathbf{1}_{n-1}^{\intercal}\\
\hline
\alpha\mathbf{1}_{n-1}&\mathbf{1}_{n-1} & nI_{n-1}+J_{n-1}
\end{array}
\right]
\simeq
\left[
\begin{array}{cc|c}
n+1 & \alpha^* & \mathbf{0}_{n-1}^{\intercal}\\
n\alpha & 2n & \mathbf{0}_{n-1}^{\intercal}\\
\hline
\alpha\mathbf{1}_{n-1} & \mathbf{1}_{n-1} & nI_{n-1}
\end{array}\right].
\end{equation}

Taking determinants, and using the fact that $|\alpha|^2=\alpha\alpha^*=n+1+2\re(s)$, we find
\begin{equation}
|\det M|^2=(2(n+1)-|\alpha|^2)n^{n}=(n+1-2\re(s))n^n.
\end{equation}
From the fact that $\re(s)\leq \sqrt{n}$, it follows that
\begin{equation}
|\det M|^2\geq (n+1-2\sqrt{n})n^n=(\sqrt{n}+1)^2n^n,
\end{equation}
and taking square roots the result follows. \qedhere
\end{proof}

 In particular, Proposition \ref{prop-ConstRSBoundComp} shows that the existence of a $\BH(n,\ell)$ matrix with constant row-sum, then there is a large-determinant matrix of order $n+1$ with entries in $\mu_{\ell}$. A \textit{Bush-type} Hadamard matrix is an Hadamard matrix of order $n^2$ consisting of $n^2$ blocks $E_{ij}$ of order $n$, such that $E_{ij}J_n=\delta_{ij}J_n$. In other words, a Bush-type Hadamard matrix has diagonal blocks equal to $J_n$, and all other blocks have constant row-sum equal to $0$. The following construction of Bush-type matrices is well-known, see for example Kharaghani's paper \cite{Kharaghani-BushType}.
\begin{proposition}[cf. \cite{Kharaghani-BushType}] \normalfont\label{thm-BushTypeConst} Let $H$ be a dephased $\BH(n,\ell)$. Let $r_i$ be the $i$-th row of $H$, and let $E_i=r_i^*r_i$ be the rank-$1$ projection matrix onto the subspace spanned by $r_i$. Then the block-circulant matrix $M=[E_{i-j}]_{ij}$, i.e.
\begin{equation}
M=\begin{bmatrix}
E_0 & E_1& E_2 & \dots &E_{n-1}\\
E_{n-1} & E_0 & E_1 & \dots & E_{n-2}\\
E_{n-2} & E_{n-1} & E_0 & \dots & E_{n-3}\\
\vdots & \vdots & \vdots & \ddots & \vdots\\
E_{1} & E_2 & E_3 & \dots & E_0
\end{bmatrix},
\end{equation}

 is a Bush-type $\BH(n^2,\ell)$ matrix.
\end{proposition}

Combining the results above we obtain the following determinantal lower bound:
\begin{theorem}\label{thm-BushLowerBound}
If there is a $\BH(n,\ell)$, then there is a matrix $M$ of order $n^2+1$ with entries in the $\ell$-th roots of unity such that
\begin{equation}|\det(M)|\geq (n+1)n^{n^2}.
\end{equation}
\end{theorem}
\begin{proof}
Let $H_0$ be a $\BH(n,\ell)$, Theorem \ref{thm-BushTypeConst} gives the existence of a Bush-type $\BH(n^2,\ell)$, say $H$. Apply Proposition \ref{prop-ConstRSBoundComp} to $H$ to obtain a matrix $M$ of order $n^2+1$ satisfying the lower bound in the statement.\qedhere
\end{proof}
\begin{corollary}\normalfont \label{cor-BushLowerBound}For all integers $\ell$, $t\geq 1$, 
\begin{equation}\gamma_{\ell}(\ell^{2t}+1)\geq (\ell^{t}+1)\ell^{t\ell^{2t}}. 
\end{equation}
\end{corollary}
\begin{proof}
The Fourier matrix $F_{\ell}$ is a $\BH(\ell,\ell)$ matrix. Taking the $t$-fold tensor product of $F_{\ell}$ with itself we obtain a $\BH(\ell^t,\ell)$ for all $t\geq 1$. Then, Theorem \ref{thm-BushLowerBound} with $n=\ell^t$ yields the result.\qedhere
\end{proof}
Notice that 
\begin{equation}
\lim_{t}\frac{\gamma_{\ell}(\ell^{2t}+1)}{h(\ell^{2t}+1)}\geq\lim_{t}\frac{(\ell^t+1)\ell^{t\ell^{2t}}}{(\ell^{2t}+1)^{(\ell^{2t}+1)/2}}=\frac{1}{\sqrt{e}},
\end{equation}
so our construction achieves asymptotically at least 60\% of the Hadamard bound. In the cases $\ell=2,3,4$, the ratio of our determinantal lower bound to the Barba bound in Corollary \ref{cor-Barba3} is $1/\sqrt{2}$, which is a fraction of approximately 70\%.

%%%%%%%%%%%%%%%%%%%%%%%%%%%
% THIRD ROOTS LOWER BOUND %
%%%%%%%%%%%%%%%%%%%%%%%%%%%
\section{Lower bounds for matrices over the third roots}\label{sec-3Roots}
The existence of a $\BH(n,3)$ implies that $3\mid n$. So, over the third roots of unity, the lower bound in Theorem \ref{thm-BushLowerBound} applies only at orders $\equiv 1\pmod{3}$. In this section, we give an additional family of large determinant matrices over $\mu_3$ at orders $q+1\equiv 2\pmod{3}$, where $q$ is a prime power.

\begin{definition}
    Let $\ell>1$ be an integer, and $q$ be a prime power. Suppose that $q\equiv 1\pmod{\ell}$, and let $\chi$ be a  character of $(\F_{q}^{\times},\cdot)$ of order $\ell$. Extend $\chi$ to $\F_q$ by letting $\chi(0)=0$. We define the \textit{generalized Paley core} $Q_{q}$ as
    \begin{equation}
        (Q_{q})_{ij} = \chi(i-j),
    \end{equation}
    for every $i,j\in\F_q$
\end{definition}
The entries of $Q_q$ are $\ell$-th roots of unity, and $Q_q$ satisfies the following matrix identities:
\begin{itemize}
    \item [(i)] $Q_qQ_q^*=qI_{q}-J_q,$ and
    \item [(ii)] $Q_qJ_q=\mathbf{0}$.
\end{itemize}

Notice that the usual Paley core \cite{Horadam-HadamardBook}, corresponds to the case $\ell=2$. The extended matrix
\begin{equation}\label{eq-WeighingMatrix}
    W_q = \begin{bmatrix}
    0 & \mathbf{1}_{q}^{\intercal}\\
    \mathbf{1}_{q}^{\intercal} & Q_q
    \end{bmatrix},
\end{equation}
satisfies $W_qW_q^*=qI_{q+1}$, i.e. it is a \textit{generalized weighing matrix} $\GW_{\ell}(q+1,q)$ (see Theorem 5.2.13. \cite{Ponasso-Thesis}). The diagonal of $W_q$ is equal to the zero vector, so it cannot directly yield a lower bound for the determinant of a matrix over the $\ell$-th roots.  We can, however, compute the determinant of $W_q+\alpha I_q$, where $\alpha\in\mu_{\ell}$, using the theory of \textit{cyclotomy}, \cite{Hall-CombinatorialTheoryBook, Storer-CyclotomyBook}. This will, in turn, give us a determinantal lower bound at orders $q+1\equiv 2\pmod{3}.$

\begin{definition}\normalfont Let $q$ be a prime power, and let $\ell$ and $f$ be integers such that $q=\ell f+1$. The \textit{$\ell$-th cyclotomic classes} are the cosets of the subgroup $H_0$ of $\ell$-th powers in $\F_q^{\times}$. In other words, the cyclotomic classes are the sets 
\begin{equation}
H_i=\{\gamma^{\ell a+i}: a\in \{0,1\dots,f-1\}\},
\end{equation}
for $0\leq i\leq \ell-1$, where $\gamma$ is a primitive element of $\F_q$.
\end{definition} 

\begin{definition}\normalfont Over the field $\F_q$, for $q$ a prime power, the \textit{$\ell$-th cyclotomic number} $(i,j)$ is defined as the number of elements $x_i\in H_i$ such that
\begin{equation}
x_i+1\in H_j.
\end{equation}
Equivalently, the cyclotomic number $(i,j)$ is the number of solutions $(x,y)$ to the equation 
\begin{equation}
    \gamma^{\ell x+i}+1=\gamma^{\ell y+j}.
\end{equation}
\end{definition}

The cyclotomic numbers satisfy the following elementary relations:
\begin{lemma}[cf. Part 1, Lemma 3 \cite{Storer-CyclotomyBook}]\label{lemma-CyclotomicNumsSymmetry} \normalfont Let $(i,j)$ be the $\ell$-th cyclotomic number, where $q=\ell f+1$. Then
\begin{enumerate}
\item $(i+a\ell,j+b\ell)=(i,j)$ for any integers $a,b$,
\item $(i,j)=(\ell-i,j-i)$,
\item $(i,j)=(j,i)$ if $f$ is even, and $(i,j)=(j+\ell/2,i+\ell/2)$ if $f$ is odd,
\item
\[\sum_{j=0}^{\ell-1}(i,j)=
\begin{cases}
f-1 & \text{ if } f\text{ is even, and } i=0\\
f-1 & \text{ if } f\text{ is odd, and } i=\ell/2\\
f & \text{ otherwise }
\end{cases}
\]
\item
\[\sum_{i=0}^{\ell-1}(i,j)=
\begin{cases}
f-1 & \text{ if } j=0\\
f & \text{ otherwise }
\end{cases}
\]
\end{enumerate}
\end{lemma}

The following Lemma characterizes \textit{cubic cyclotomy}, i.e. the case $\ell=3$, in terms of $q$.
\begin{theorem}[Cubic cyclotomic numbers, cf. Part 1, Lemma 7 \cite{Storer-CyclotomyBook}]\label{thm-CubicCyclotomy}
The cyclotomic numbers are given by
\begin{equation}
\begin{bmatrix}
    (0,0) & (0,1) & (0,2)\\
    (1,0) & (1,1) & (1,2)\\
    (2,0) & (2,1) & (2,2)
\end{bmatrix}=
\begin{bmatrix}
    A & B & C\\
    B & C & D\\
    C & D & B
\end{bmatrix},
\end{equation}
where 
\begin{equation}
\begin{split}
9A &= q-8+c\\
18B&=2q-4-c-9d\\
18C&=2q-4-c+9d\\
9D &= q+1+c,
\end{split}
\end{equation}
and $c$ is uniquely determined by the conditions $4q=c^2+27d^2$, and $c\equiv 1\pmod{3}$.
\end{theorem}
\begin{remark}
The cubic cyclotomic numbers $A=(0,0)$ and $D=(1,2)=(2,1)$ are uniquely determined, but $B$ and $C$ are interchanged according to the sign of $d$. We do not need to resolve this indeterminacy, since our results will be independent of the value of $d$.
\end{remark}
\begin{lemma}[cf. Part 1, Section 3 \cite{Storer-CyclotomyBook}]\normalfont \label{lemma-TripleSum} Let $q=3f+1$ be a prime power, and $H_i$ be the cubic cyclotomic classes in $\F_q$. Let $N$ be the number of solutions to the equation
\begin{equation}
    1+x_0+x_1+x_2=0,
\end{equation}
where $x_i\in H_i$; for $i=0,1,2$. Then, $N$ satisfies
\begin{equation}
N=AD+B^2+C^2=BC+BD+CD=\frac{1}{27}(q^2-3q-c).
\end{equation}
\end{lemma}

The relationship between the generalized Paley core, its eigenvalues, and cyclotomy, can be established from a representation-theoretic point of view. For this we recall the following:

\begin{definition}
    Let $\Gamma=(\F_q,+)$, and $\C[\Gamma]$ be the group algebra of $\Gamma$, i.e. the $\C$-algebra generated by formal sums $\sum_{x\in \Gamma}\alpha_x[x]$ where each $\alpha_x\in \C$. For each $x,y\in \Gamma$, the product of  $[x]$ and $[y]$ in $G$ is defined as $[x]\cdot [y]:=[x+y]$. An element $\alpha\in \F_q^{\times}$ induces an automorphism of the group $\Gamma=(\F_q,+)$ by left multiplication: $x\mapsto \alpha x$, which induces an automorphism of $\C[\Gamma]$ written as $[x]^{\alpha}:=[\alpha x]$.
\end{definition}

Define 
\begin{equation}
K_i:=\sum_{x\in H_i}[x]=\sum_{a=0}^{f-1}[\gamma^{a\ell+i}]\in\C[\Gamma],
\end{equation}
for $0\leq i\leq \ell-1$. Let $\varrho$ be the regular representation of $\Gamma$, i.e. $\varrho(z)$ for $z\in\F_q$ is the $q\times q$ matrix indexed by elements $x,y$ in $\F_q$ such that
\begin{equation}
(\varrho(z))_{xy}=\begin{cases}
    1 & \text{ if } x+z = y\\
    0 & \text{ otherwise }
\end{cases}.
\end{equation}
Extending $\varrho$ to $\C[\Gamma]$ by linearity, the generalized Paley core can be written as
\begin{equation}
    Q_q = \varrho\left(\sum_{i=0}^{\ell-1}\zeta_q^{i}K_i\right).
\end{equation}

Linear relationships for the product of matrices $\varrho(K_i)$ and $\varrho(K_j)$ can be computed using the cyclotomic numbers:

\begin{proposition}[cf. \cite{Munemasa-Cyclotomic}]\normalfont \label{prop-CycloConstants} Let $H_r$ be the unique $\ell$-th cyclotomic class containing the element $-1$. The product in $\C[\Gamma]$ of $K_i$ and $K_j$ is 
\begin{equation} K_iK_j=f\delta_{i',j}[0]+\sum_{k=0}^{\ell-1}(j-i,k-i)K_k,
\end{equation}
where $i'=i+r$.
\end{proposition}

This proposition shows that the matrices $A_i:=\varrho(K_i)$ are the adjacency matrices of the \textit{Bose-Mesner algebra} of an $\ell$-class \textit{association scheme} \cite{Bannai-Ito}, called the \textit{cyclotomic association scheme}. \\ 

The following is a straightforward fact, see also \cite{Babai-FourierAbelianGroups}:
\begin{lemma}\label{lemma-RegularRepDiagonalization}
  Let $\Gamma=\{g_0=1_{G},g_1,\dots,g_{n-1}\}$ be a finite abelian group of order $n$. Let $\{\chi_0,\chi_1,\dots,\chi_{n-1}\}$ be the irreducible characters of $\Gamma$, where $\chi_i:=\hat{g_i}$ is the image of $g_i$ under the isomorphism $\hat{(\cdot)}:\Gamma\rightarrow \hat{\Gamma}$ between $\Gamma$ and its dual. Let $X$ be the character table of $\Gamma$, indexed as $X_{ij}=\chi_i(g_j)$. Then, for every $g\in G$,
  \[\varrho(g)X= X\diag(\chi_0(g),\chi_1(g),\dots,\chi_{n-1}(g)).\]
\end{lemma}
In other words, the character table of $\Gamma$ diagonalizes the regular representation of $\Gamma$, and the eigenvalues of $\varrho(g)$ are given by evaluating $g$ at every irreducible character of $\Gamma$. Since $Q_{q}+\alpha I_q$ is a linear combination of the matrices $\varrho(K_i)$, Lemma \ref{lemma-RegularRepDiagonalization} gives us the following:
\begin{corollary}\label{cor-RNFDet}
    Let $Q_q$ be a generalized Paley matrix of order $q$. Then,
    \begin{equation}
    \det (W_q+\alpha I_{q+1})=
        \det \begin{bmatrix}
            \alpha & \mathbf{1}_q^{\intercal}\\
            \mathbf{1}_q & Q_q +\alpha I_q
        \end{bmatrix}
        =\frac{\alpha^2-q}{\alpha}\det(Q_q+\alpha I_q).
    \end{equation}
\end{corollary}
\begin{proof}
    To prove this, we find the rational normal form of the matrix $W_q+\alpha I_{q+1}$. The weighing matrix $W_q$ can be decomposed as $W_q=A+B$, where
    \begin{equation}
    A=\begin{bmatrix}
        0 & \mathbf{1}_q^{\intercal}\\
        \mathbf{1}_q & \mathbf{0}
    \end{bmatrix},\text{ and }
    B=
    \begin{bmatrix}
        0 & \mathbf{0}_q^{\intercal}\\
        \mathbf{0}_q & Q_q
    \end{bmatrix}.
    \end{equation}
     Let $X$ be the character table of $\Gamma=(\F_q,+)$, and let 
    \begin{equation}
        F = \begin{bmatrix}
            1 & \mathbf{0}_q^{\intercal}\\
            \mathbf{0}_q & X
        \end{bmatrix}.
    \end{equation}
    Since $X^{-1}=\frac{1}{q}X^*$, we have:

    \begin{equation}
        F^{-1}AF = \begin{bmatrix}
            0 & \mathbf{1}_q^{\intercal}X\\
            X^{-1}\mathbf{1}_q &\mathbf{0}
        \end{bmatrix}
        =
        \left[
        \begin{array}{cc|c}
            0 & q & \mathbf{0}_{q-1}^{\intercal}\\
            1 & 0 & \mathbf{0}_{q-1}^{\intercal}\\
            \hline
            \mathbf{0}_{q-1} & \mathbf{0}_{q-1} & \mathbf{0}
        \end{array}\right].
    \end{equation}
    Lemma \ref{lemma-RegularRepDiagonalization} implies that $X^{-1}Q_q X = \begin{bmatrix}0 & \mathbf{0}_{q-1}^{\intercal}\\ \mathbf{0}_q &\Delta\end{bmatrix}$, where $\Delta$ is the diagonal matrix given by the non-zero eigenvalues of $Q_q$. Hence,
    \begin{equation}
        F^{-1}BF = \left[\begin{array}{cc|c}
        0 & 0 & \mathbf{0}_{q-1}^{\intercal}\\
        0 & 0 & \mathbf{0}_{q-1}^{\intercal}\\
        \hline
        \mathbf{0}_{q-1} & \mathbf{0}_{q-1} & \Delta
        \end{array}\right].
    \end{equation}
    From the above, we conclude that
\begin{equation}
    \det \begin{bmatrix}
            \alpha & \mathbf{1}_q^{\intercal}\\
            \mathbf{1}_q & Q_q +\alpha I_q
        \end{bmatrix} = \det(W_q+\alpha I_{q+1})=\det \left[\begin{array}{cc|c}
        \alpha & q & \mathbf{0}_{q-1}^{\intercal}\\
        1 & \alpha & \mathbf{0}_{q-1}^{\intercal}\\
        \hline
        \mathbf{0}_{q-1} & \mathbf{0}_{q-1} & \Delta+\alpha I_{q-1}
        \end{array}\right].
\end{equation}
Computing the determinant by blocks, and noticing that $\det(\Delta+\alpha I_{q-1})=\frac{1}{\alpha}\det(Q_{q}+\alpha I_q)$, we obtain the result.\qedhere
\end{proof}

Corollary \ref{cor-RNFDet} allows us to reduce the computation of $\det(W_q+\alpha I_{q+1})$ to the computation of $\det(Q_q+\alpha I_{q})$. The following is an auxiliary result to compute this latter determinant:

\begin{proposition}\normalfont\label{prop-GramComputationPaley} Let $q=3f+1$ be a prime power, and $\alpha\in \C$. The Gram matrix $G=(Q_q+\alpha I_q)(Q_q+\alpha I_q)^*$, is expressed as the image under $\varrho$ of an element of $\C[\Gamma]$ as follows:
\begin{equation} 
G=\varrho\left((|\alpha|^2+q-1)[0]+(2\re(\alpha) -1)K_{0}+(2\re(\alpha\omega^2)-1)K_{1}+
(2\re(\alpha\omega)-1)K_{2}\right).
\end{equation}
\end{proposition}
\begin{proof}
 The matrix $Q_q+\alpha I_q $ can be written as:
\begin{equation} Q_q+\alpha I_q = \varrho(\alpha[0]+ K_0+\omega K_1+\omega^2 K_2).
\end{equation}
Notice that $\rho(g)^{\intercal}=\rho(-g)$. Since $-1\in H_0$, we have that $-H_i=H_i$, and the matrix $(Q_q+\alpha I_q)^*$ is given by the element $\overline{\alpha}[0]+K_0+\omega^2 K_1 + \omega K_2\in\C[\Gamma]$. Computing the product in $\C[\Gamma]$ we find:
\begin{equation}
\begin{split}
&(\alpha [0]+ K_0+\omega K_1+\omega^2 K_2)(\overline{\alpha} [0] + K_0+ \omega^2 K_1+ \omega K_2)\\
= &|\alpha|^2[0]+\alpha K_0+\alpha\omega^2 K_1+\alpha \omega K_2\\
&\overline{\alpha}K_0+K_0^2 + \omega^2 K_0K_1+\omega K_0K_2\\
&\overline{\alpha}\omega K_1+\omega K_1K_0+K_1^2+\omega^2 K_1K_2\\
&\overline{\alpha}\omega^2 K_2+\omega^2 K_2K_0+\omega K_2K_1+K_2^2.
\end{split}
\end{equation}
We evaluate this expression: First we find by Proposition \ref{prop-CycloConstants} and Lemma \ref{lemma-CyclotomicNumsSymmetry}, that
\begin{equation}
K_0^2+K_1^2+K_2^2 = 3f[0] +\sum_k \left(\sum_i(0,k-i)\right)K_k =(q-1)f[0]+(f-1)(K_0+K_1+K_2).
\end{equation}
Next we evaluate $\omega(K_0K_2+K_1K_0+K_2K_1)$ and $\omega^2(K_0K_1+K_1K_2+K_2K_0)$. It is easy to check that
\begin{equation}
\begin{split}
K_0K_2+K_1K_0+K_2K_1&=f(K_0+K_1+K_2),\text{ and}\\
K_0K_1+K_1K_2+K_2K_0&=f(K_0+K_1+K_2).
\end{split}
\end{equation}
This implies that $\omega(K_0K_2+K_1K_0+K_2K_1)+\omega^2(K_0K_1+K_1K_2+K_2K_0)=f(\omega+\omega^2)(K_0+K_1+K_2)=-f(K_0+K_1+K_2).$ Therefore the Gram matrix is given by the image of $\varrho$ at the element
\begin{equation}
(|\alpha|^2+(q-1))[0]+(2\re(\alpha)-1)K_0+(2\re(\alpha\omega^2)-1)K_1+(2\re(\alpha\omega)-1)K_2,
\end{equation}
  as we wanted to show.\qedhere
\end{proof}

Now, Lemma \ref{lemma-RegularRepDiagonalization} will allow us to compute the eigenvalues of the Gram matrix of $Q_q+\alpha I_q$. These eigenvalues can be expressed using the following:
\begin{definition}\normalfont Let $q=\ell f+1$ be a prime power. The $\ell$-th \textit{Gaussian periods} are defined as 
\[\eta_i:=\sum_{x\in H_i} \zeta_q^{x},\ \ \  0\leq i\leq \ell-1.\]
\end{definition}

\begin{corollary}\label{cor-EigValsGramPaley} \normalfont Let $q=3f+1$ be a prime power, and let $Q_q$ be the generalised Paley core of order $q$ over the third roots. Then, the eigenvalues of the Gram matrix of $Q_q+\alpha I_q$ with $\alpha\in\{1,\omega,\omega^2\}$ are:
\begin{itemize}
\item[(i)] $1=q-3f$, occurring with multiplicity $1$, and
\item[(ii)] $(q+2)+3\eta_i$, each occurring with multiplicity $f$, for $i=0,1,2$.
\end{itemize}
\end{corollary}
\begin{proof}
By Proposition \ref{prop-GramComputationPaley}, and using the fact that $\alpha\in \{1,\omega,\omega^2\}$, we have that $M=(Q_q+\alpha I_q)(Q_q+\alpha I_q)^*$ is given by the group algebra element
\begin{equation}
\begin{split}
&(|\alpha|^2+(q-1))[0]+(2\re(\alpha)-1)K_0+(2\re(\alpha\omega^2)-1)K_1+(2\re(\alpha\omega^2)-1)K_2\\
&=q[0]+K_{i}-2K_{i+1}-2K_{i+2},
\end{split}
\end{equation}
for some $i=0,1,2$. Lemma \ref{lemma-RegularRepDiagonalization} implies that the eigenvalues of $M$ are given by the evaluation of $q[0]+K_{i}-2K_{i+1}-2K_{i+2}$ at each linear character of the additive group $\Gamma=(\F_q,+)$. For the trivial character $\chi_0$, we find
\begin{equation}
\chi_0(q[0]+K_{i}-2K_{i+1}-2K_{i+2})=q-3|K_0|=q-3f=1.
\end{equation}
All non-trivial linear characters are of the type $\chi_j(x)$, $1\leq j\leq q-1$, where $\chi_j(1)=\zeta_q^{\gamma^j}$, and $\gamma$ is a primitive element of $\F_q^{\times}$. We have that 
\begin{equation}
\begin{split}
\chi_j(q[0]+K_i-2K_{i+1}-2K_{i+2})&=
q+\chi_j(K_i)-2\chi_j(K_{i+1})
-2\chi_j(K_{i+1})\\
&=q+\eta_{i+j}-2\eta_{i+j+1}-2\eta_{i+j+2}\\
&=q+3\eta_k -2(\eta_0+\eta_1+\eta_2),
\end{split}
\end{equation}
for some $k\in\{0,1,2\}$. Since $\sum_{x\in\F_  q}\zeta_q^x=0$, we have that $\eta_0+\eta_1+\eta_2=-1$, and the eigenvalues are $q+2+3\eta_i$, each with multiplicity $f=(q-1)/3$, for $i=0,1,2$.\qedhere
\end{proof}

\begin{theorem}\label{thm-PaleyDeterminant} Let $q=3f+1$ be a prime power, and $\alpha\in\{1,\omega,\omega^2\}$. Then, the absolute value of the determinant of $Q_q+\alpha I_q$ is
\begin{equation}
\begin{split}
|\det(Q_q+\alpha I_q)|&=\left(\prod_{i=0}^2((q+2)+3\eta_i)\right)^{f/2}\\
&=\left[(q+2)^3-3(q+2)^2-3(q-1)(q+2)+(3+c)q-1\right]^{(q-1)/6},
\end{split}
\end{equation}
where $c$ is uniquely determined by $4q=c^2+27d^2$, and $c\equiv 1\pmod{3}$.
\end{theorem}
\begin{proof}
By Corollary \ref{cor-EigValsGramPaley}, the determinant of $(Q_q+\alpha I_q)(Q_q+\alpha I_q)^*$ is
\begin{equation}
[(q+2+3\eta_0)(q+2+3\eta_1)(q+2+3\eta_2)]^f.
\end{equation}
By Lemma \ref{lemma-RegularRepDiagonalization}, instead of directly computing the product above, we can equivalently compute it using the images of the elements $K_i$ in the quotient ring $\C[\Gamma]/(S_{\Gamma})$, where $S_{\Gamma}:=\sum_{x\in \Gamma}[x]$. The product $\prod_{i} ((q+2)[0]+3K_i)$ expands as
\begin{equation}
\begin{split}
(q+2)^3[0]&+3(q+2)^2(K_0+K_1+K_2)\\
&+3^2(q+2)(K_0K_1+K_0K_2+K_1K_2)\\
&+3^3K_0K_1K_2.
\end{split}
\end{equation}
We have that $K_0K_1+K_0K_2+K_1K_2=f(K_0+K_1+K_2)\equiv -f[0] \pmod{S_{\Gamma}}$, so
\begin{equation}
\prod_{i=0}^2((q+2)[0]+3K_i)\equiv ((q+2)^3 -3(q+2)^2-3^2(q+2)f)[0]+ 3^3K_0K_1K_2\pmod{S_{\Gamma}}.
\end{equation}
It remains to compute the term $K_0K_1K_2$. Using the notation of Theorem \ref{thm-CubicCyclotomy}, we have
\begin{equation}
\begin{split}
K_0K_1K_2&= K_0(DK_0+BK_1+CK_2)\\
&=fD\cdot [0]+ADK_0+BDK_1+CDK_2\\
&\phantom{=fD\cdot [0]\ }
+B^2K_0+BC K_1+ BDK_2\\
&\phantom{=fD\cdot [0]\ }
+C^2K_0+CDK_1+CBK_2.
\end{split}
\end{equation}
By Lemma \ref{lemma-TripleSum}, we have that $AD+B^2+C^2=BC+BD+CD=N=(q^2-3q-c)/3^3$.
Therefore,
\begin{equation}
\begin{split}
K_0K_1K_2&=fD[0]+N(K_0+K_1+K_2)\\
&\equiv (fD-N)[0]\pmod{S_{\Gamma}}.
\end{split}
\end{equation}
Substituting $N=(q^2-3q-c)/3^3$, and using the fact that $f=(q-1)/3$ and $3^2D=(q+1+c)$, we have that $3^2(q+2)f=3(q-1)(q+2)$ and $3^3(fD-N)=3q-1+qc$. Therefore,
\begin{equation}
\prod_{i=0}^2((q+2)[0]+3K_i)\equiv ((q+2)^3 -3(q+2)^2-3(q-1)(q+2)+(3+c)q-1)[0]\pmod{S_{\Gamma}}.
\end{equation}
Evaluating this expression at a non-trivial character of $\F_q$, the result follows.\qedhere
\end{proof}
\begin{corollary}\normalfont For every prime power $q\equiv 1\pmod{3}$, there is a matrix $M$ of order $q+1\equiv 2\pmod{3}$ over the third roots of unity such that
\begin{equation}\label{eq-PrimePowerBound}
|\det M|=\sqrt{q^2+q+1}\cdot \left[(q+2)^3-3(q+2)^2-3(q-1)(q+2)+(3+c)q-1\right]^{(q-1)/6},
\end{equation}
where $c$ is uniquely determined by $4q=c^2+27d^2$, and $c\equiv 1\pmod{3}$.
\end{corollary}
\begin{proof}
By Corollary \ref{cor-RNFDet}, the determinant of $W+\alpha I_{q+1}$ is
\begin{equation}
  |\det(W+\alpha I_{q+1})|=\left|\frac{\alpha^2-q}{\alpha}\right||\det(Q_q+\alpha I_q)|.
\end{equation}
If $\alpha=1$, then $|\alpha^2-q|/|\alpha|=q-1$. On the other hand, if $\alpha\in\{\omega,\omega^2\}$ then $|\alpha-q|/|\alpha|=\sqrt{q^2+q+1}>q-1$, so the largest value of the determinant is obtained with $\alpha=\omega$ or $\alpha=\omega^2$. Applying Theorem \ref{thm-PaleyDeterminant} to $Q_q+\omega I_q$, we find that the determinant of $M=W+\omega I_{q+1}$ satisfies Equation \ref{eq-PrimePowerBound}.\qedhere
\end{proof}
In particular, this gives an infinite family of matrices of order $n\equiv 2\pmod{3}$ that achieve a constant ratio of the Barba bound. We have
\begin{equation}
  \lim_{q\rightarrow\infty} \frac{\sqrt{q^2+q+1}\cdot \left[(q+2)^3-3(q+2)^2-3(q-1)(q+2)+(3+c)q-1\right]^{(q-1)/6}}{\sqrt{2q+1}\cdot q^{q/2}}=\frac{1}{\sqrt{2}}.
  \end{equation}
So our construction achieves approximately $70\%$ of the Barba bound in the limit.

\begin{example}
  Let $q=4\equiv 1\pmod{3}$, and write $\F_4=\{0,1,t,t+1\}$, where $t^2=t+1$. Then $H_0=\{1\}$, $H_1=\{t\}$, and $H_2=\{t+1\}$; the generalized Paley matrix is:
  \begin{equation}
    Q_4 = \begin{bmatrix}
      0 & 1 & \omega & \omega^2\\
      1 & 0 & \omega^2 & \omega\\
      \omega & \omega^2 & 0 & 1\\
      \omega^2 & \omega & 1 & 0
      \end{bmatrix},
  \end{equation}
  and the matrix of our construction is
  \begin{equation}\label{eq-M5}
    M_5 = \begin{bmatrix}
      \omega & \mathbf{1}_n^{\intercal}\\
      \mathbf{1}_n & Q_4+\omega I_4
    \end{bmatrix}
    =
    \begin{bmatrix}
      \omega & 1 & 1 & 1 & 1\\
      1 & \omega & 1 & \omega & \omega^2\\
      1 & 1 & \omega & \omega^2 & \omega\\
      1 & \omega & \omega^2 & \omega & 1\\
      1 &\omega^2 & \omega  & 1 & \omega
      \end{bmatrix}.
  \end{equation}
  It is easy to check that $|\det(M_5)|=\sqrt{1701}$. Later we will show (see Table \ref{tab-Search5}) that this matrix is of maximal determinant over the third roots.
\end{example}

\begin{example}
  Let $q=7\equiv 1\pmod{3}$, and write $\F_7=\{0,1,2,3,4,5,6\}$. Then $H_0=\{1,6\}$, $H_1=\{3,4\}$, and $H_2=\{2,5\}$; the generalized Paley matrix is:
  \begin{equation}
    Q_7 = \begin{bmatrix}
      0 & 1 & \omega^2 & \omega & \omega & \omega^2 & 1\\
      1 & 0 & 1 & \omega^2 & \omega & \omega & \omega^2\\
      \omega^2 & 1 & 0 & 1 & \omega^2 & \omega & \omega\\
      \omega & \omega^2 & 1 & 0 & 1 & \omega^2 & \omega\\
      \omega & \omega & \omega^2 & 1 & 0 & 1 & \omega^2\\
      \omega^2 & \omega & \omega & \omega^2 & 1 & 0 & 1\\
      1 & \omega^2 & \omega & \omega & \omega^2 & 1 & 0
      \end{bmatrix},
  \end{equation}
  and the matrix given by our construction is,

  \begin{equation}
    W_8+\omega I_8 = \begin{bmatrix}
      \omega & \mathbf{1}_n^{\intercal}\\
      \mathbf{1}_n & Q_7+\omega I_7
    \end{bmatrix}
    =
    \begin{bmatrix}
      \omega & 1 & 1 & 1 & 1 & 1 & 1 & 1\\
      1 & \omega & 1 & \omega^2 & \omega & \omega & \omega^2 & 1\\
      1 & 1 & \omega & 1 & \omega^2 & \omega & \omega & \omega^2\\
      1 & \omega^2 & 1 & \omega & 1 & \omega^2 & \omega & \omega\\
      1 & \omega & \omega^2 & 1 & \omega & 1 & \omega^2 & \omega\\
      1 & \omega & \omega & \omega^2 & 1 & \omega & 1 & \omega^2\\
      1 & \omega^2 & \omega & \omega & \omega^2 & 1 & \omega & 1\\
      1 & 1 & \omega^2 & \omega & \omega & \omega^2 & 1 & \omega
      \end{bmatrix}.
    \end{equation}
  We find $|\det(W_8+\omega I_8)|=\sqrt{7022457}$ and, in this case, the matrix is not maximal determinant over the third roots. Instead, it achieves $\approx 78.4\%$ of the maximal value.
\end{example}

%%%%%%%%%%%%%%%%%%%%%%%%%%%%%%%%%%%%%%%%%%%%%%%%%%%%
%%% === SECTION - SMALL MAXIMAL DETERMINANTS === %%%
%%%%%%%%%%%%%%%%%%%%%%%%%%%%%%%%%%%%%%%%%%%%%%%%%%%%

\section{Maximal determinant matrices of small order over the third roots}\label{sec-3RootsSmall}
In this section we given an account of maximal determinant values for matrices over the third roots of unity. To this end, we first summarize the known results on the existence of $\BH(n,3)$ matrices with $n\equiv 0\pmod{3}$, then we give examples of Barba matrices of small orders $n\equiv 1\pmod{3}$, and finally we give certificates of determinant maximality for matrices of order $n\equiv 2\pmod{3}$ that do not meet the Hadamard or Barba bounds. These results are summarized in Table \ref{tab-Maxdet3}.

\begin{table}[h]
  \centering
  \caption{Largest known determinant values for matrices of order $n$ over the third roots: The columns labelled $n$ contain the matrix order. The columns labelled $|\det|^2/3^{n-1}$ contain the absolute value squared of the largest known determinant value at order $n$, divided by $3^{n-1}$. Columns labelled R contain the ratio of the largest known determinant absolute value to the Hadamard bound for orders $n\equiv 1\pmod{3}$, and to the Barba bound for orders $n\equiv 1,2\pmod{3}$. We use the symbol '\red{??}' to indicate that we have no proof of maximality for the given determinant value. If said symbol does not appear, the corresponding determinant value has been proven to be maximal.}\label{tab-Maxdet3}
  
  \begin{tabular}{*{3}{|lcr|}}
    \hline
    $n$ & $|\det|^2/3^{n-1}$ & R & $n$ & $|\det|^2/3^{n-1}$ & R & $n$ & $|\det|^2/3^{n-1}$ & R \\
    \hline
      & & & $1$ & $1$ & $1$ & $2$ & $1$ & $1$\\
    $3$ & $3$ & $1$ & $4$ & $7$ & $1$ & $5$ & $3\times 7$& $0.86$\\
    $6$ & $2^6\times 3$ & $1$ & $7$ & $2^6\times 13$ & $1$ & $8$ & $2^{12}$ & $0.85$\\
    $9$ & $3^{10}$ & $1$ & $10$ & $3^9\times 19$ & $1$ & $11$ & $3^9\times 7\times 19$ & $0.86$\\
    $12$ & $2^{24}\times 3$ & $1$ & $13$ & $2^{24}\times 5^2$ & $1$ & $14$ & $2^{24}\times 223$\red{??} & \red{$0.85$}\\
    $15$ & $2^{22}\times 3^6\times 19$\red{??} & \red{$0.79$} & $16$ & $2^{24}\times 3^{8}\times 7$\red{??} & \red{$0.90$} & $17$ & $13^{5}\times 67^{4}$\red{??} & \red{$0.72$}\\
    $18$ & $2^{18}\times 3^{19}$ & 1 & $19$ & $13\times 37^2 \times 342037^2$\red{??} & \red{$0.74$} & $20$ & $7^6\times 37^6 \times 127$ \red{??} & \red{$0.76$}\\ \hline
    \end{tabular}
  \end{table}

\subsection{Butson matrices over the third roots}
Butson-type Hadamard matrices over the third roots, $\BH(n,3)$, can only exist at orders $n\equiv 0\pmod{3}$. There are two general constructions based on tensor products that we can use to generate $\BH(n,3)$ matrices. See the survey on Butson-type Hadamard matrices on Chapter 4 of \cite{Ponasso-Thesis}.

\begin{enumerate}
\item The tensor product $A\otimes B$ of a $\BH(n,3)$ and a $\BH(m,3)$ is a $\BH(nm,3)$ \textemdash this is Sylvester's construction \cite{Sylvester-InverseOrthogonal}.
\item If there is a $\BH(n,3)$ and a projective plane of order $n-1$, then there is a $\BH((n-1)n,3)$\textemdash this is a generalization of Scarpis' construction, cf. \cite{Scarpis} and see Theorem 4.3.3 in \cite{Ponasso-Thesis}.
\end{enumerate}

Warwick de Launey generalized Scarpis' result for to show the existence of $\BH(2^t\cdot 3,3)$ matrices for all $t\geq 0$, \cite{deLauney-Thesis}. In particular it shows the existence of a $\BH(6,3)$ matrix, recovering the result of Butson \cite{Butson} on the existence of $\BH(2p,p)$ matrices for all primes $p$ in the case $p=3$. These explain the entries in Table \ref{tab-Maxdet3} at orders $n\equiv 0\pmod{3}$: using the fact that the Fourier matrix $F_3$ is a $\BH(3,3)$ and the results above, we can construct Butson-type Hadamard matrices $\BH(6,3), \BH(9,3), \BH(12,3),$ and $\BH(18,3)$. Using the non-existence results of Winterhof \cite{Winterhof-NonexistenceButson}, one can show that there is no $\BH(15,3)$\textemdash see also Chapter 3 of \cite{Ponasso-Thesis}. The open cases for the existence of a $\BH(n,3)$ matrix with $n<100$ are  $n=39, 42, 57, 60, 66, 75, 78, 84,$ and $93$.

\subsection{Barba matrices over the third roots}
Barba matrices over the third roots only exist at orders $n\equiv 1\pmod{3}$. Recall that for an element $\alpha=a+b\omega\in\Z[\omega]$ the \textit{norm} of $\alpha$ in the quadratic ring extension $\Z\subset\Z[\omega]$ is given by $N(\alpha):=\alpha\overline{\alpha}=a^2-ab+b^2$. The proposition below imposes restrictions on the prime factors of $N(\alpha)$, and can be used to obtain additional non-existence conditions for Barba matrices:

\begin{proposition}\label{prop-Norm3Conditions} Let $\alpha=a+b\omega\in\Z[\omega]$, and let $N(\alpha):=\alpha\overline{\alpha}=a^2-ab+b^2\in \Z$ be the norm of $\alpha$. Suppose that $p\neq 3$ is a prime number dividing the square-free part of $N(\alpha)$. Then, $p\equiv 1\pmod{3}$.
\end{proposition}
\begin{proof}
  Recall from algebraic number theory, e.g. Theorem 25 in \cite{Marcus}, that, in the quadratic ring extension $\Z\subset\Z[\omega]$, the assumption $p\neq 3$ implies two possibilities for the prime ideal decomposition of the ideal $(p)$ in $\Z[\omega]$:
  \begin{enumerate}
  \item $(p)$ is prime and \textit{self-conjugate}, in the sense that $\overline{(p)}=(p)$ \textemdash the \textit{inert} case; or
  \item $(p)=\mathfrak{p}\mathfrak{q}$, where $\mathfrak{p}$ and $\mathfrak{q}$ are \textit{conjugate} prime ideals in $\Z[\omega]$, i.e. $\mathfrak{q}=\overline{\mathfrak{p}}$ \textemdash the \textit{split} case.
  \end{enumerate}
  Since $\Z[\omega]$ is a Dedekind domain, the prime ideal factorization of $(\alpha\overline{\alpha})=(\alpha)\overline{(\alpha)}$ is unique. Suppose $(p)$ divides $(\alpha)$, then $\overline{(p)}$ divides $\overline{(\alpha)}$ and $(p)$ must have the decomposition $(p)=\mathfrak{p}\mathfrak{q}$ with $\mathfrak{p}=\overline{\mathfrak{q}}$; otherwise, $p$ cannot divide the square-free part of $N(\alpha)$ \textemdash the multiplicity with which $p$ divides $\alpha\overline{\alpha}$ is even. The prime $(p)$ splits as $(p)=\mathfrak{p}\mathfrak{q}$ in $\Z[\omega]$ if and only if $-3$ is a square residue modulo $p$, i.e. if and only if $\legendre{-3}{p}=+1$; see Theorem 25 \cite{Marcus}. By quadratic reciprocity $\legendre{-3}{p}=\legendre{p}{3}$ and $(p)$ splits if and only if $p\equiv 1\pmod{3}$.
\end{proof}

\begin{theorem}\label{thm-NonExistenceBarba3} Let $n\equiv 1\pmod{3}$ be an integer, and write $(2n-1)(n-1)^{n-1}=s^2r$ with $r$ square-free. If $p\equiv 2\pmod{3}$ is a prime number dividing $r$, then there is no Barba matrix of order $n$ over the third roots.
\end{theorem}
\begin{proof}
Suppose that a Barba matrix $B$ of order $n$ exists, then $\det(M)\in\Z[\omega]$ and,
  \begin{equation}
    N(\det(M))=\det(M)\overline{\det(M)} = (2n-1)(n-1)^{n-1}=s^2r.
  \end{equation}
By Proposition \ref{prop-Norm3Conditions}, if $p\equiv 2\pmod{3}$ is a prime number dividing $r$, we have a contradiction; so a Barba matrix of order $n$ over the third roots cannot exist.
\end{proof}

By Theorem \ref{thm-NonExistenceBarba3}, the first few orders $n\equiv 1\pmod{3}$ at which a Barba matrix over the third roots cannot exist are:
\[n=16, 28, 34, 43, 46, 52, 58, 70, 73, 88, 94, 100, 103, 106, 118, 124, 127, 133, 136, 142, 148, \dots
\]
For example, $n=16$ fails since in this case $(2n-1)(n-1)^{n-1}=s^2(31\cdot 5\cdot 3)$ for some integer $s$, and the prime factor $5$ is congruent to $2$ modulo $3$.\\

For orders $n<19$, examples of Barba matrices over the third roots can be obtained from a symmetric 2-design, or a strongly-regular graph.\\

The matrix,
\begin{equation}
    B_4 = \begin{bmatrix}
        \omega & 1 & 1 & 1\\
        1 & \omega & 1 & 1\\
        1 & 1 & \omega & 1\\
        1 & 1 & 1 & \omega
    \end{bmatrix}
\end{equation}
is a Barba matrix, associated to the trivial $2$-design in $4$-points. Let $D$ be an incidence matrix of the Fano plane, then $B_7=(\omega-1)D + J_7$ is a Barba matrix of order $7$, for example:
\begin{equation}
B_7=
\begin{bmatrix}
    1&\omega&\omega&1&\omega&1&1\\
    1&1&\omega&\omega&1&\omega&1\\
    1&1&1&\omega&\omega&1&\omega\\
    \omega&1&1&1&\omega&\omega&1\\
    1&\omega&1&1&1&\omega&\omega\\
    \omega&1&\omega&1&1&1&\omega\\
    \omega&\omega&1&\omega&1&1&1
\end{bmatrix}.
\end{equation}
It can be shown (Corollary 5.3.1. \cite{Ponasso-Thesis}) that, up to equivalence, these two are the only Barba matrices over the third roots having exactly two distinct entries. Likewise, it can be shown (Corollary 5.3.2. \cite{Ponasso-Thesis}) that the only two strongly regular graphs whose Bose-Mesner algebra supports a Barba matrix over the third roots are the Petersen graph and the Paley graph of order $13$. The associated Barba matrices\textemdash written in logarithmic notation\textemdash are:

\begin{equation}
{\tiny
    B_{10} = \left[
\begin{array}{*{10}{c}}
0&2&1&2&1&1&2&2&2&2\\
2&0&2&2&2&1&2&1&1&2\\
1&2&0&1&2&2&2&2&1&2\\
2&2&1&0&2&2&2&1&2&1\\
1&2&2&2&0&2&1&1&2&2\\
1&1&2&2&2&0&2&2&2&1\\
2&2&2&2&1&2&0&2&1&1\\
2&1&2&1&1&2&2&0&2&2\\
2&1&1&2&2&2&1&2&0&2\\
2&2&2&1&2&1&1&2&2&0
\end{array}\right],
    B_{13}=\left[\begin{array}{*{13}{c}}
0&1&2&2&1&2&2&1&1&1&1&2&2\\
1&0&2&2&2&1&1&1&1&2&2&1&2\\
2&2&0&1&2&1&2&2&1&1&1&1&2\\
2&2&1&0&2&2&1&1&1&1&2&2&1\\
1&2&2&2&0&2&1&2&2&1&1&1&1\\
2&1&1&2&2&0&2&1&2&2&1&1&1\\
2&1&2&1&1&2&0&2&1&2&2&1&1\\
1&1&2&1&2&1&2&0&2&1&2&2&1\\
1&1&1&1&2&2&1&2&0&2&1&2&2\\
1&2&1&1&1&2&2&1&2&0&2&1&2\\
1&2&1&2&1&1&2&2&1&2&0&2&1\\
2&1&1&2&1&1&1&2&2&1&2&0&2\\
2&2&2&1&1&1&1&1&2&2&1&2&0
\end{array}\right].
}
\end{equation}

At the time of writing, the smallest open case allowed by Theorem \ref{thm-NonExistenceBarba3} is $n=19$.

\subsection{Maximal determinant matrices at orders \texorpdfstring{$n\equiv 2\pmod{3}$}{n ≡ 2 (mod 3)}}

In this subsection, we develop tools to give certificates of maximality of putative maximal determinant matrices of order $n\equiv 2\pmod{3}$. At orders $n\equiv 2\pmod{3}$, the bounds of Theorem \ref{thm-Hadamard} and Corollary \ref{cor-Barba3} cannot be met with equality, so computational approachs similar to the ones developed in \cite{Moyssiadis-Kounias} and \cite{Orrick-15} are needed.\\

Let $\tilde{X}$ be a candidate maximal determinant matrix over the third roots, and let $\tilde{G}=\tilde{X}\tilde{X}^*$ be its Gram matrix.
The algorithm we will describe is a backtracking search for the set of all putative Gram matrices $G=XX^*$ of matrices $X$ with entries in the set $\{1,\omega,\omega^2\}$ satisfying $\det(G)\geq\det(\tilde{G})$. In particular, we will begin with a set $\mathcal{G}_1=\{[n]\}$ of $1\times 1$ principal submatrices of Gram matrices, and at the $r$-th step of the algorithm we will construct a set $\mathcal{G}_r$ of principal $r\times r$ submatrices from the set $\mathcal{G}_{r-1}$: this is done by considering all feasible extensions of a matrix $G\in\mathcal{G}_{r-1}$ of the type
\begin{equation}G_f:=\left[\begin{array}{c|c}
    G & f\\
    \hline
    f^* & n
    \end{array}\right],
\end{equation}
where $f$ is a vector of length $r-1$. If the final set of candidate matrices $\mathcal{G}:=\mathcal{G}_n$ contains no matrices $G$ with determinant strictly greater than that of $\tilde{G}$, then $\tilde{X}$ is of maximal determinant.\\

In order to reduce the size of the backtracking search tree, we carry some pruning steps. First, we can ``normalize'' matrices over the third-roots, to reduce the number of possible off-diagonal entries to consider.

\begin{definition} Let $v$ be a vector of length $n$ with entries in $\{1,\omega,\omega^2\}$, and let $v_1$, $v_{\omega}$ and $v_{\omega^2}$ the number of entries in $v$ equal to $1$, $\omega$, and $\omega^2$, respectively. We say that $v$ is \textit{balanced} if  $v_{\omega}\equiv v_{\omega^2}\pmod{3}$. We say that a square matrix of order $n$ with entries in $\{1,\omega,\omega^2\}$ is \textit{balanced} if all of its row vectors and column vectors are balanced.
\end{definition}

\begin{lemma}\label{lemma-BalancedVectors}
  Let $v$ be a balanced vector, and let $\Delta=\diag(\omega^{e_1},\dots,\omega^{e_n})$ where $\sum_{i=1}^n e_i \equiv 0\pmod{3}$. Then, the vector $\Delta v$ is also balanced.
\end{lemma}
\begin{proof} Let $u=(\omega^{e_1},\dots,\omega^{e_n})^{\intercal}$, then the product $\Delta v$ is equal to the entry-wise product $u\circ v$ of $u$ and $v$. Without loss of generality, we may assume that $u$ and $v$ have the following form:
  \begin{equation}\label{eq-uvWlog}
    \begin{split}
      v &= [\mathbf{1}_{v_1} | \omega\mathbf{1}_{v_\omega} | \omega^2\mathbf{1}_{v_{\omega}^2}]; \text{ and }\\
      u &= [\mathbf{1}_{a_1}|\omega\mathbf{1}_{a_{\omega}}|\omega^2\mathbf{1}_{a_{\omega^2}}|\mathbf{1}_{b_1}|\omega\mathbf{1}_{b_{\omega}}|\omega^2\mathbf{1}_{b_{\omega^2}}|\mathbf{1}_{c_1}|\omega\mathbf{1}_{c_{\omega}}|\omega^2\mathbf{1}_{c_{\omega^2}}],
    \end{split}
  \end{equation}
  where $a_1+a_{\omega}+a_{\omega^2}=v_1$, $b_1+b_{\omega}+b_{\omega^2}=v_{\omega}$, and $c_1+c_{\omega}+c_{\omega^2}=v_{\omega^2}$. Since $v$ is balanced, we have that
  \begin{equation}\label{eq-vBalanced}
    b_1+b_{\omega}+b_{\omega^2}\equiv c_1+c_{\omega}+c_{\omega^2}\pmod{3},
  \end{equation}
  and from the hypothesis $\sum_{i=1}^n e_i\equiv 0\pmod{3}$, we find
  \begin{equation}
    0(a_1+b_1+c_1)+1(a_{\omega}+b_{\omega}+c_{\omega})+2(a_{\omega^2}+b_{\omega^2}+c_{\omega^2})\equiv 0\pmod{3}. 
  \end{equation}
  From this, it follows that
  \begin{equation}\label{eq-uHypothesis}
    a_{\omega}+b_{\omega}+c_{\omega} \equiv a_{\omega^2}+b_{\omega^2}+c_{\omega^2}\pmod{3}.
  \end{equation}
  Now, combining Equation \ref{eq-vBalanced} with Equation \ref{eq-uHypothesis}, we find that
  \begin{equation}
    \begin{split}
      (\Delta v)_{\omega}&=a_{\omega}+b_1+c_{\omega^2}\\
      &\equiv a_{\omega}+(c_1+c_{\omega}+c_{\omega^2}-b_{\omega}-b_{\omega^2})+c_{\omega^2}\\
      &\equiv a_{\omega}-b_{\omega}-b_{\omega^2}+c_1+c_{\omega}-c_{\omega^2}\\
      &\equiv (a_{\omega^2}+b_{\omega^2}+c_{\omega^2}-b_{\omega}-c_{\omega})-b_{\omega}-b_{\omega^2}+c_1+c_{\omega}-c_{\omega^2}\\
      &\equiv a_{\omega^2}+b_{\omega}+c_1  \pmod{3}.
    \end{split}
  \end{equation}
  Since $(\Delta_v)_{\omega^2}=a_{\omega^2}+b_{\omega}+c_1$, this shows that $(\Delta_v)_{\omega}\equiv (\Delta_v)_{\omega^2}\pmod{3}$, i.e. $\Delta v$ is balanced.\qedhere
\end{proof}

\begin{lemma}\label{lemma-BalancedMatrices} Let $M$ be a matrix of order $n\equiv 1,2\pmod{3}$ with entries in the set $\{1,\omega,\omega^2\}$. Then, there exists a unique pair of diagonal matrices $(\Delta_1,\Delta_2)$ with diagonal entries in $\{1,\omega,\omega^2\}$ such that $\Delta_1 M\Delta_2$ is balanced.
\end{lemma}
\begin{proof} We begin with the following observation: Let $(v_{1},v_{\omega},v_{\omega^2})\in \Z^3$ be a triple, then its unordered list  of residue classes modulo $3$, $[v_{1}\pmod{3},v_{\omega}\pmod{3},v_{\omega^2}\pmod{3}]$, is one of the following:
  \[
  \begin{array}{ccc}
    000 & 001 & 002\\
    012 & 022 & 011\\
    111 & 112 & 122\\
    222
    \end{array}
  \]
  Under the assumption that $v_1+v_{\omega}+v_{\omega^2}=n\equiv 1,2\pmod{3}$, the list of residue classes belongs to the second or third columns above. Furthermore, we see that every list in the second and third column has exactly two identical elements. This implies that for an arbitrary vector $v$ of length $n\equiv 1,2\pmod{3}$, there exists a unique $i$ such that the vector $\omega^i v$ is balanced.\\
  Let $M$ be a matrix of order $n$; by the observation above, there is a unique diagonal matrix $\Delta_1$ such that every row of $\Delta_1 M$ is balanced. Likewise, there is a unique diagonal matrix $\Delta_2$ such that all the columns of $\Delta_1M\Delta_2$ are balanced. It suffices to show that all the rows of $\Delta_1M\Delta_2$ are balanced: To see this, let $N:=\Delta_1M$ and let $N_{1}$, $N_{\omega}$, and $N_{\omega^2}$ be the total number of entries of $N$ equal to $1$, $\omega$, and $\omega^2$, respectively. Since all rows of $N$ are balanced, $N_{\omega}\equiv N_{\omega^2}\pmod{3}$. Let $c_1,\dots,c_n$ be the columns of $N$, and define $e_i\in\{0,1,2\}$ by $e_i\equiv (c_i)_{\omega}-(c_i)_{\omega^2}\pmod{3}$, then
  \begin{equation}                                                         
    \sum_{i=1}^n e_i \equiv \sum_{i=1}^n\left[(c_i)_{\omega}-(c_i)_{\omega^2}\right]=N_{\omega}-N_{\omega^2}\equiv 0\pmod 3.                          
    \end{equation}                                                         
  Since $\Delta_2$ is unique we must have $\Delta_2=\diag(\omega^{e_1},\dots,\omega^{e_n})$, and by Lemma \ref{lemma-BalancedVectors} we find that every row of $N\Delta_2=\Delta_1M\Delta_2$ is balanced.\qedhere
\end{proof}

\begin{corollary}\label{cor-Equivalence} Let $X_1$ and $X_2$ be two balanced matrices of order $n$. Then, $X_1$ and $X_2$ are equivalent under the action of the group of $n\times n$ monomial matrices over the third roots if and only if $X_1$ and $X_2$ are permutation-equivalent.
\end{corollary}
\begin{proof}
  Suppose that there exists a pair of permutation matrices $(P,Q)$ and a pair of diagonal matrices $(\Delta_1,\Delta_2)$ with diagonal entries in $\{1,\omega,\omega^2\}$ such that
  \begin{equation}
    X_1= \Delta_1PX_2Q\Delta_2.
  \end{equation}
  Then, since the matrix $PX_2Q$ is balanced, Lemma \ref{lemma-BalancedMatrices} implies that $\Delta_1=\Delta_2=I_n$. Therefore $X_1=PX_2Q$, i.e. $X_1$ and $X_2$ are permutation-equivalent. \qedhere
\end{proof}

\begin{remark} In view of Lemma \ref{lemma-BalancedMatrices} and Corollary \ref{cor-Equivalence}, we may restrict our search to Gram matrices of balanced matrices up to permutation equivalence: Let $X_1$ and $X_2$ be balanced matrices of the same order $n$, and let $G_1=X_1X_1^*$, $G_2=X_2X_2^*$ be their respective Gram matrices. If $X_1$ and $X_2$ are permutation-equivalent, then there exists a pair $(P,Q)$ of permutation matrices such that $X_1=PX_2Q$. Therefore
  \begin{equation}
    G_1= X_1X_1^*= (PX_2Q)(PX_2Q)^*=PX_2X_2^*P^{\intercal}=PG_2P^{\intercal}.
  \end{equation}
  So $G_1$ can be obtained from $G_2$ by a series of symmetric row-column permutations.
\end{remark}

\begin{remark} The equivalence of Hermitian matrices under symmetric row-column permutations can be expressed as a graph isomorphism problem, and tested computationally using the \text{C} library \texttt{nauty} \cite{McKay-GraphIsoII}: Let $\mathcal{G}$ be a set of $k\times k$ Hermitian matrices with entries taken from a set $\Phi=\{g_1,\dots,g_{f}\}$. For every $G\in\mathcal{G}$, we define a vertex-colored graph $\mathcal{P}(G)$ in the following way:
  \begin{itemize}
  \item $\mathcal{P}(G)$ has a total of $2k+f+k^2$ vertices: the first $k$ vertices $\{r_1,\dots,r_k\}$ correspond to each of the rows of $G$; the next $k$ vertices $\{c_1,\dots,c_k\}$ to each of the columns of $G$; the following $f$ vertices $\{g_1,\dots,g_f\}$ correspond to each of the possible distinct entries of $G\in\mathcal{G}$; and the remaining $k^2$ entries correspond to each of the possible pairs $(i,j)$, $1\leq i,j\leq k$.
  \item Every vertex $(i,j)$ is joined by an edge to $r_i$ and $c_j$, and $(i,j)$ is joined by an edge to $g_r$ if and only if $G_{ij}=g_r$.
  \item The color classes of the graph $\mathcal{P}(G)$ are $\{r_1,\dots,r_k\}$, $\{c_1,\dots,c_k\}$, $\{(i,j):1\leq i,j\leq k\}$ and $\{g_r\}$ for $1\leq r\leq k$.
  \end{itemize}
  See Figure \ref{fig-PermutationGraphExample} for an example of a graph $\mathcal{P}(G)$.
\end{remark}

\begin{figure}
  \centering
  \begin{tikzpicture}\def\yshift{1.25}

     \node at (-6,0.5) {$ G=
        \left[ \begin{matrix}
            \textcolor{orange}{\boldsymbol{n}} & \textcolor{magenta}{\boldsymbol{-1-3\omega}} \\
            \textcolor{green}{\boldsymbol{2+3\omega}} & \textcolor{orange}{\boldsymbol{n}} \\
        \end{matrix} \right]
        $};

     \draw[thick,->, line width=1mm] (-3,0.5) --  (-1.5,0.5)
      node[midway, above=2mm] {$\mathcal{P}(\cdot)$};
    
    % DRAW ROW NODES
    \node[draw, fill=red, scale=1.5, label=left:{$r_1$}](r1) at (0,1) {};
    \node[draw, fill=red, scale=1.5, label=left:{$r_2$}](r2) at (0,0) {};

    % DRAW COLUMN NODES
    \node[draw, regular polygon, regular polygon sides=3, scale=0.8, fill=blue, label=above:{$c_1$}](c1) at (2,3) {};
    \node[draw, regular polygon, regular polygon sides=3, scale=0.8, fill=blue, label=above:{$c_2$}](c2) at (3,3) {};

    % DRAW COLOR NODES
    \node[draw, circle, fill=orange, scale=1.2, label=right:{$n$}](n) at (5,\yshift+1){};
    \node[draw, regular polygon, regular polygon sides=6, scale=1.2, fill=cyan, label=right:{$-1$}](g1) at (5,\yshift){};
    \node[draw, diamond, scale=1, fill=magenta, label=right:{$-1-3\omega$}](g2) at (5,\yshift -1){};
    \node[draw, regular polygon, regular polygon sides=7, scale=1.2, fill=green, label=right:{$2+3\omega$}] (g3) at (5,\yshift -2){};

    % DRAW GRID
    \foreach \i in {1,2} {
      \foreach \j in {1,2} {
        \node[draw, fill=black, circle, scale=0.5](p\i\j) at (1+\i,-1+\j) {};  % Create circle nodes at (i,j)
      }
    }

    % Warning: the labels pij do not correspond to the matrix indexing!
    \draw[thick] (p11) to (r2); % p11 = G_21
    \draw[thick, bend left] (p11) to (c1); 
    \draw[thick] (p12) to (r1); % p12 = G_11
    \draw[thick] (p12) to (c1);
    \draw[thick, bend left] (p21) to (r2); % p21 = G_22
    \draw[thick, bend right] (p21) to (c2);
    \draw[thick, bend right] (p22) to (r1); % p22 = G_12
    \draw[thick] (p22) to (c2);

    \draw[thick, bend left] (p12) to (n);
    \draw[thick, bend right] (p21) to (n);

    \draw[thick, bend right] (p11) to (g3);
    \draw[thick, bend left] (p22) to (g2);
    
  \end{tikzpicture}
  \caption{An example of a graph $\mathcal{P}(G)$ obtained from a matrix $G\in \mathcal{G}$, where $\mathcal{G}$ is a set of $2\times 2$ Hermitian matrices with entries in $\Phi=\{n,-1,-1-3\omega,2+3\omega\}$. Notice that in the rightmost column of nodes we have included the node labeled $-1$, even if this entry does not appear in the matrix $G$.}\label{fig-PermutationGraphExample}
\end{figure}
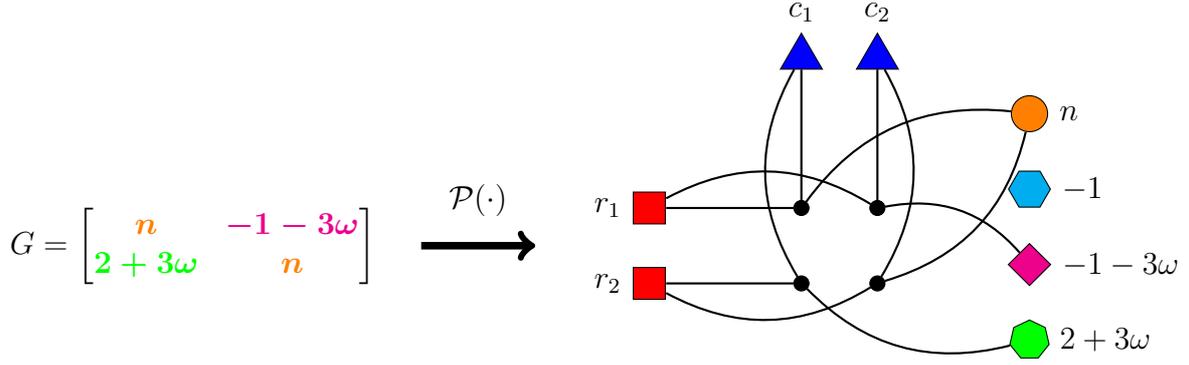

 The following lemma reduces the number of off-diagonal entries we need to consider:

\begin{lemma}\label{lemma-BalancedInnerProducts} Let $u$ and $v$ be two balanced vectors of length $n$. Then
  \begin{equation}
    u\cdot v = a+3b\omega,
  \end{equation}
  with $a,b\in \Z$ and $a\equiv n\pmod{3}$.
\end{lemma}
\begin{proof}
  Without loss of generality we may assume that $u$ and $v$ have the form given in Equation \ref{eq-uvWlog}. Then, the inner product of $u$ and $v$ can be written as:
  \begin{equation}
    u\cdot v = a_1+b_{\omega^2}+c_{\omega}-a_{\omega^2}-b_{\omega}-c_1 + \omega(a_{\omega}+b_1+c_{\omega^2}-a_{\omega^2}-b_{\omega}-c_{1}).
  \end{equation}
  From the fact that $u$ and $v$ are balanced, we have that
  \begin{align}
    b_1+b_{\omega}+b_{\omega^2}&\equiv c_1+c_{\omega}+c_{\omega^2}\pmod{3}; \text{ and }\\
    a_{\omega}+b_{\omega}+c_{\omega} &\equiv a_{\omega^2}+b_{\omega^2}+c_{\omega^2}\pmod{3}.
  \end{align}
  From these two equations, it follows that
  \begin{equation}
    a_{\omega}+b_1+c_{\omega^2}-a_{\omega^2}-b_{\omega}-c_1 \equiv b_1+b_{\omega}+b_{\omega^2}-(c_1+c_{\omega}+c_{\omega^2})\equiv 0\pmod{3}.
  \end{equation}
  We can then write $u\cdot v = a+3b\omega$ with $a,b\in\Z$. In general, we can write $u\cdot v = x+y\omega+z\omega^2 = (x-z)\omega+(y-z)\omega$ for some integers $x,y,z\geq 0$ satisfying $x+y+z=n$. From the above, we have $y-z\equiv 0\pmod{3}$, and
  \begin{equation}
    a=x-z=n-y-2z\equiv n-3z\equiv n\pmod{3}, 
  \end{equation}
  as we wanted to show.\qedhere
\end{proof}

An additional pruning method is given by the following: 
\begin{definition}[cf. Definition A.1. \cite{MK-21}]  \label{def-SForm}
  An Hermitian matrix $M$ of order $n$ is in \textit{standard form} if and only if
  \begin{enumerate}
  \item $|M_{2i-1,2i}| \geq |M_{j,k}|$ for all $i\equiv 1\pmod{2}$ and for all $2i-1\leq j< k\leq n$; and
  \item $|M_{i,i+1}|\geq \max\{|M_{i,i+2}|,|M_{i-1,i+1}|,|M_{i-1,i+2}|\}$ for all $i\equiv 0\pmod{2}$.
  \end{enumerate}
\end{definition}
\begin{remark}
  It is easy to check that in the process of generating all principal submatrix sets $\mathcal{G}_r$ we may assume that all matrices in $\mathcal{G}_k$ are in standard form.
\end{remark}

The main pruning technique we employ is an adaptation of the Moyssiadis-Kounias bound \cite{Moyssiadis-Kounias, MK-21}, and Orrick's sharpening in \cite{Orrick-15}. The following are extensions of these bounds to hermitian matrices; the proofs are analogous but we reproduce them here for completeness:

\begin{theorem}[cf. Moyssiadis-Kounias \cite{Moyssiadis-Kounias}]\label{thm-MKBound}
    Let $D$ be an $r\times r$ Hermitian matrix, and $M$ be an $m\times m$ Hermitian positive-definite matrix extending $D$ as follows:
    \begin{equation}
        M=\begin{bmatrix}
            D & B\\
            B^* & A
        \end{bmatrix},
    \end{equation}
   Assume furthermore that $M_{ii}=n$, and that the entries of $M$ are taken from a set $\Phi\subseteq\C$ such that $|\alpha|\geq c$ for all $\alpha\in \Phi$, and let
    \begin{equation}
        \hat{d}:=\max_{\gamma\in\Phi^{r}}\det\begin{bmatrix}
            D & \gamma\\
            \gamma^* &c
        \end{bmatrix}.
    \end{equation}
    Then
    \begin{equation}\label{eq-MKBound}
        \det M\leq (n-c)^{m-r-1}\left[(n-c)\det D+(m-r)\max(0,\hat{d})\right].
    \end{equation}
\end{theorem}
\begin{proof}
    We prove this by induction on $m-r$. When $m-r=1$, any  matrix $M$ of order $m$ extending $D$ is of the shape
    \begin{equation}
        M=\begin{bmatrix}
            D & \gamma\\
            \gamma^* & n
        \end{bmatrix},
    \end{equation}
    where $\gamma\in \Phi^r$. By the linearity of the determinant on rows:
    \begin{equation}
    \begin{split}
        \det M &= (n-c)\det D + \det\begin{bmatrix}
            D & \gamma\\
            \gamma^* & c
        \end{bmatrix}\\
        &\leq (n-c)\det D + \max(0,\hat{d}).
    \end{split}
    \end{equation}
    Suppose the claim is true for $k=m-r$, we show that the claim is true for all Hermitian, positive-definite matrices of order $m=k+r+1$ extending $D$. Such a matrix $M$ has the shape
    \begin{equation}
        M=\begin{bmatrix}
            D & B_1 & \gamma\\
            B_1^*&A_1&\delta\\
            \gamma^*&\delta^*&n
        \end{bmatrix}.
    \end{equation}
    By linearity of the determinant,
    \begin{equation}
    \begin{split}
        \det M &= (n-c)\begin{bmatrix}
            D & B_1\\
            B_1^*&A_1
        \end{bmatrix}
        + \det \begin{bmatrix}
            D & B_1 & \gamma\\
            B_1^* & A_1 & \delta\\
            \gamma^* & \delta^* & c
        \end{bmatrix}\\
        &=(n-c) \begin{bmatrix}
            D & B_1\\
            B_1^*&A_1
        \end{bmatrix} + c\det\begin{bmatrix}
        D-\gamma\gamma^*/c & B_1-\gamma\delta^*/c\\
        B_1^*-\delta\gamma^*/c & A_1-\delta\delta^*/c
    \end{bmatrix}
    \end{split}
    \end{equation}
    Let
    \begin{equation}
      M'=\begin{bmatrix}
      D & B_1 & \gamma\\
      B_1^*&A_1&\delta\\
      \gamma^*&\delta^*&c
      \end{bmatrix}
    \end{equation}
    If $\det M'\leq 0$, then the claim is trivially true. Otherwise, $M'$ is positive-definite, which implies that $A_1-\delta\delta^*/c$ is also positive-definite. By Fischer's inequality
    \begin{equation}
        c\det\begin{bmatrix}
        D-\gamma\gamma^*/c & B_1-\gamma\delta^*/c\\
        B_1^*-\delta\gamma^*/c & A_1-\delta\delta^*/c
    \end{bmatrix}\leq c\det(D-\gamma\gamma^*/c) \det(A_1-\delta\delta^*/c).
    \end{equation}
    By definition, $c\det(D-\gamma\gamma^*/c)\leq \hat{d}$. Since $A-\delta\delta^*/c$ is positive-definite and has diagonal terms bounded above by $(n-c)$,
    \begin{equation}
      \det(A_1-\delta\delta^*/c)\leq (n-c)^{m-r}.
    \end{equation}
    Now, applying the induction hypothesis we have
    \begin{equation}
      \begin{split}
        \det M &\leq (n-c)^{m-r}\left[(n-c)\det D+(m-r)\max(0,\hat{d})\right]+\max(0,\hat{d})(n-c)^{m-r}\\
        &=(n-c)^{m-r}\left[(n-c)\det D+(m-r+1)\max(0,\hat{d})\right],
      \end{split}
    \end{equation}
    which is the claim we wanted to prove. \qedhere
\end{proof}

\begin{corollary}[cf. Orrick \cite{Orrick-15}]\label{cor-MKB1}
  Let $M$ be a Hermitian positive-definite matrix of order $m$ with $M_{ii}\leq n$, and $|M_{ij}|\geq c$. Then,
  \begin{equation}\label{eq-MKB1}
    \det M \leq (n-c)^m +cm(n-c)^{m-1}.
    \end{equation}
\end{corollary}
\begin{proof}
  Apply the Moyssiadis-Kounias bound with $D=[n]$, and $r=1$. We have that $\hat{d}=c(n-c)$, since
  \begin{equation}
    \det\begin{bmatrix}
    n & \alpha\\
    \overline{\alpha} & c
    \end{bmatrix}
    = nc -|\alpha|^2\leq c(n-c).
  \end{equation}
  Substituting the values of $D$, $r$, and $\hat{d}$ in Equation \ref{eq-MKBound} yields the result.
\end{proof}

The following result can be proved in analogy to Corollary 3.2 in \cite{Orrick-15}.

\begin{corollary}\label{cor-MKBn2}
  Let $M=\begin{bmatrix}D & B\\B^* & A\end{bmatrix}$ be an $n\times n$ matrix satisfying the conditions of Theorem \ref{thm-MKBound}, and suppose furthermore that the entries of $M$ are in $\Z[\omega]$, and that $M_{ij}\equiv n\equiv 2\pmod{(1-\omega)}$. Then,
  \begin{equation}
    \det M\leq (n-1)^{n-r}\det D +\max(0,\hat{d})\left[2(n-1)^{n-r}-2(n-2)^{n-r}-(n-r)(n-2)^{n-r-1}\right].
  \end{equation}
\end{corollary}
\begin{proof}
  In the proof of Theorem \ref{thm-MKBound} we used the fact that if $\det\overline{M}>0$, then $A_1-\delta\delta^*/c$ is positive-definite and used Fischer's inequality to conclude
  \begin{equation}
    \det(A_1-\delta\delta^*/c)\leq (n-c)^{m-r-1}.
  \end{equation}
  Corollary \ref{cor-MKB1} gives us a stronger upper bound for the determinant of $A_1-\delta\delta^*/c$: First, notice that, under our assumptions, $c=1$ and that $\delta_i\equiv 2\pmod{(1-\omega)}$ for all $i$. This implies
  \begin{equation}
    (A-\delta\delta^*/c)_{ij}= A_{ij} - \delta_i\overline{\delta_j}\equiv 2-2\cdot 2\equiv 1\pmod{(1-\omega)}.
  \end{equation}
  Also, since $A_{ii}=n$, we have that $(A-\delta\delta^*)_{ii}\leq n-1$. Applying Inequality \ref{eq-MKB1}, we find
  \begin{equation}
    \det(A-\delta\delta^*)\leq (n-2)^{m-r-1}+(m-r-1)(n-2)^{m-r-2}.
  \end{equation}
  Using this inequality in the proof of Theorem \ref{thm-MKBound}, we can prove by induction that
  \begin{equation}\label{eq-MKBE-0}
      \det M \leq (n-1)^{m-r}\det(D)+\max(0,\hat{d}) S_{m-r},
  \end{equation}
  where $S_{k}=\sum_{j=0}^{k-1}(n-1)^{k-1-j}(n-2)^j + \sum_{j=0}^{k-2}(n-1)^{k-2-j}(j+1)(n-2)^{j}$. Using the formula for the partial sum of a geometric series, and its derivative, we can prove that
  \begin{equation}\label{eq-MKSumComputation}
    S_k = 2(n-1)^{k}-2(n-2)^k-k(n-2)^{k-1}.
  \end{equation}
  Using Equation \ref{eq-MKSumComputation} to compute $S_{m-k}$, and substituting this value in Equation \ref{eq-MKBE-0} we find,
  \begin{equation}
    \det M\leq (n-1)^{m-r}\det D +\max(0,\hat{d})\left[2(n-1)^{m-r}-2(n-2)^{m-r}-(m-r)(n-2)^{m-r-1}\right].
  \end{equation}
  Letting $m=n$, we obtain the sought inequality.
  \end{proof}

We are now ready to present our algorithm to generate candidate Gram matrices. Let $n\equiv 1,2\pmod{3}$ be the order of our candidate maximal determinant matrix $\tilde{X}$. Let $\tilde{G}=\tilde{X}\tilde{X}^*$, and let $\Phi_n=\{a+3b\omega:a,b\in\Z;~ a\equiv n\pmod{3},~ |a+b\omega|<n\}$ be \textemdash by means of Lemma \ref{lemma-BalancedInnerProducts}\textemdash the set of admissible off-diagonal elements. The algorithm proceeds as follows:

\begin{enumerate}
\item Start with the set $\mathcal{G}_1=\{[n]\}$ of putative $1\times 1$ Gram matrices; assign $r\leftarrow 1$ and $\Phi_n^{(1)}=\Phi_n$.
\item Generate the set $\mathcal{G}_{r+1}$ from the set $\mathcal{G}_{r}$ by extending each matrix $G\in\mathcal{G}_{r}$ by each vector $f\in\Phi_n^{(r)}$ in the following way:
  \[
  G_f := \left[
    \begin{array}{c|c}
      G & f\\
      \hline
      f^* & n
  \end{array}
  \right].
  \]
  To prune the search, we may select only those vectors $f$ that ensure that $G_f$ is in standard form (see Definition \ref{def-SForm}). Additionally, calculate the bound of Theorem \ref{thm-MKBound} (or the bound of Corollary \ref{cor-MKBn2}, if $n\equiv 2\pmod{3}$) for $G_f$; if this bound is less than $\tilde{d}=\det(\tilde{G})$, discard $G_f$; otherwise include $G_f$ in the list $\mathcal{G}_{r+1}$.
\item If $r+1=n$, test using the methods of Theorem \ref{thm-NonExistenceBarba3} if each matrix $G\in\mathcal{G}_{n}$ satisfies $\det(G)\in N(\Z[\omega])$; if one such $G$ does not satisfy this condition, remove it from $\mathcal{G}_n$. Return $\mathcal{G}_n$ and examine this set manually to assess whether or not $\tilde{X}$ is maximal determinant.\\
  If $r+1<n$, then construct the graph $\mathcal{P}(G)$ for every matrix $G\in\mathcal{G}_{r+1}$ and determine the isomorphism classes of such graphs. Prune $\mathcal{G}_{r+1}$ by keeping exactly one representative $G$ of each isomorphism class of graphs $\mathcal{P}(G)$. Now let $\Phi_n^{(r+1)}$ be the set of off-diagonal entries of all matrices in $\mathcal{G}_{r+1}$. Update the variable $r\leftarrow r+1$ and go to Step 2.
\end{enumerate}

\begin{remark} Since the algorithm described here only needs the value of $|\det(\tilde{G})|$ to perform the certificate of maximality, it can also be used to obtained sharpened upper bounds for the maximal determinant. To do this, we may give any determinant value $d$ as a lower bound (not necessarily $d=|\det(\tilde{G})|$); if the set of putative matrices $\mathcal{G}_n$ is empty, then we conclude that $d$ is strictly larger than the determinant of any matrix over the third roots of order $n$.
\end{remark}

\begin{example} We show that the following $8\times 8$ matrix, written in exponential form,
  \begin{equation}
    \tilde{M}_8=\left[\begin{array}{cc|cc||cc|cc}
    2 & 0 & 0 & 2 & 1 & 1 & 0 & 0\\
    0 & 2 & 2 & 0 & 1 & 1 & 0 & 0\\
    \hline
    0 & 2 & 0 & 2 & 0 & 0 & 1 & 1\\
    2 & 0 & 2 & 0 & 0 & 0 & 1 & 1\\
    \hline
    \hline
    0 & 0 & 1 & 1 & 2 & 0 & 2 & 0\\
    0 & 0 & 1 & 1 & 0 & 2 & 0 & 2\\
    \hline
    1 & 1 & 0 & 0 & 2 & 0 & 0 & 2\\
    1 & 1 & 0 & 0 & 0 & 2 & 2 & 0
    \end{array}\right],
  \end{equation}
  is maximal determinant over the third roots. The Gram matrix of $\tilde{M}_8$ is
  \begin{equation}\label{eq-GramM8}
    \tilde{G}_8 = \tilde{M}_8\tilde{M}_8^* = \tilde{M}_8^*\tilde{M}_8=
    \begin{bmatrix}
      8 & 2 & - & - & - & - & - & -\\
      2 & 8 & - & - & - & - & - & -\\
      - & - & 8 & 2 & - & - & - & -\\
      - & - & 2 & 8 & - & - & - & -\\
      - & - & - & - & 8 & 2 & - & -\\
      - & - & - & - & 2 & 8 & - & -\\
      - & - & - & - & - & - & 8 & 2\\
      - & - & - & - & - & - & 2 & 8
    \end{bmatrix},
  \end{equation}
  and has determinant $G_8 = 8957952 = 2^{12}\cdot 3^7$. The Gram matrix search algorithm described above begins with the $1\times 1$ matrix $[8]$, and with the initial set of off-diagonal entries
  \begin{equation}
    \Phi_8^{(1)} = \{-7-6\omega,-4,-1+6\omega,-4-3\omega,-1+3\omega,-4-6\omega,-1,2+6\omega,-1-3\omega,2+3\omega,2,2-3\omega,5+3\omega,5\}.
  \end{equation}
  At the step $r=2$ of the search, we generate all possible $2\times 2$ submatrices
  \begin{equation}
    G_2(\alpha)=\begin{bmatrix}
    8 & \alpha\\
    \overline{\alpha} & 8
    \end{bmatrix},
  \end{equation}
  where $\alpha\in\Phi_0$. After performing all pruning steps described, only three matrices remain:
  \begin{equation}
    G_2^{(1)}=\begin{bmatrix}
    8 & -\\
    - & 8
    \end{bmatrix},~
    G_2^{(2)}=\begin{bmatrix}
    8 & 2 \\
    2 & 8
    \end{bmatrix},~
    G_2^{(3)}=
    \begin{bmatrix}
      8 & -1-3\omega\\
      2+3\omega & 8
    \end{bmatrix}.
  \end{equation}
  Hence we let $\mathcal{G}_2=\{G_2^{(1)},G_2^{(2)},G_2^{(3)}\}$, and the corresponding set $\Phi_8^{(2)}$ of admissible off-diagonal entries is
  \begin{equation}
    \Phi_8^{(2)}=\{-1,2,-1-3\omega,2+3\omega\}.
  \end{equation}
  Notice that the matrix $G_2^{(3)}$ is permutation-equivalent to $\overline{G_2^{(3)}}$. Proceeding with $\overline{G_2^{(3)}}$ instead of $G_2^{(3)}$ we would arrive at the same (final) set of $8\times 8$ candidates $\mathcal{G}_8$, although the intermediate sets $\mathcal{G}_k$ for $k<8$ may differ. The number of distinct matrices and entries in each step of the algorithm is described in Table \ref{tab-Search8}. At the final stage $r=8$ the the set of candidate matrices is $\mathcal{G}_8=\{G_8^{(1)},G_8^{(2)},G_8^{(3)},G_8^{(4)}\}$, where $G_8^{(1)}=\tilde{G}_8$ (see Equation \ref{eq-GramM8}),
  \begin{equation}
    G_8^{(2)}=
    \begin{bmatrix}
      8 & 2 & - & - & - & - & 2 & 2\\
      2 & 8 & - & - & - & - & 2 & 2\\
      - & - & 8 & 2 & - & - & - & -\\
      - & - & 2 & 8 & - & - & - & -\\
      - & - & - & - & 8 & 2 & - & -\\
      - & - & - & - & 2 & 8 & - & -\\
      2 & 2 & - & - & - & - & 8 & 2\\
      2 & 2 & - & - & - & - & 2 & 8
    \end{bmatrix},
  \end{equation}
  \begin{equation}
    G_8^{(3)}=
    \begin{bmatrix}
      8 & 2 & - & - & - & - & - & 2\\
      2 & 8 & - & - & - & - & - & 2\\
      - & - & 8 & 2 & - & - & - & -\\
      - & - & 2 & 8 & - & - & - & -\\
      - & - & - & - & 8 & 2 & - & -\\
      - & - & - & - & 2 & 8 & - & -\\
      - & - & - & - & - & - & 8 & -\\
      2 & 2 & - & - & - & - & - & 8
    \end{bmatrix},
  \end{equation}
  and,
  \begin{equation}
    G_8^{(4)}=
    \begin{bmatrix}
      8 & 2 & - & - & - & - & 2 & -\\
      2 & 8 & - & - & - & - & 2 & -\\
      - & - & 8 & 2 & - & - & - & 2\\
      - & - & 2 & 8 & - & - & - & 2\\
      - & - & - & - & 8 & 2 & - & -\\
      - & - & - & - & 2 & 8 & - & -\\
      2 & 2 & - & - & - & - & 8 & -\\
      - & - & 2 & 2 & - & - & - & 8
    \end{bmatrix}.
  \end{equation}
  The determinant of $\hat{G}_8=G_8^{(1)}$, $G_8^{(2)}$, and $G_8^{3}$ coincide, so these candidate matrices do not pose obstructions to the maximality of the determinant of $\hat{M}_8$. However,
  \begin{equation}
    \det G_8^{(4)} = 9097920 > 8957952 = \det \hat{G}_8,
  \end{equation}
  so we must show that $G_8^{(4)}$ cannot be the Gram matrix of an $8\times 8$ matrix over the third roots. This happens to be the case since $9097920=2^6\cdot 3^9\cdot 5\cdot 13$. In order for $\det G_8^{(4)}$ to be a norm of an element in $\Z[\omega]$, all prime factors $p\equiv 2\pmod{3}$ must appear with even multiplicity, however the factor $p=5$ appears with odd multiplicity. This shows that no Gram matrix of an $8\times 8$ matrix over the third roots has larger determinant than $\hat{G}_8$; hence $\hat{M}_8$ is maximal determinant.\\

  Following the same procedure, we may show that the balanced and normal $5\times 5$ matrix:
  \begin{equation}
    \tilde{M}_5=
    \left[
    \begin{array}{ccccc}
      2 & 2 & 1 & 0 & 1\\
      2 & 2 & 0 & 1 & 1\\
      1 & 0 & 2 & 2 & 1\\
      0 & 1 & 2 & 2 & 1\\
      1 & 1 & 1 & 1 & 2
    \end{array}
    \right],
  \end{equation}
  (which is equivalent to the matrix $M_5$ in Equation \ref{eq-M5}) satisfies
  \begin{equation}
    \tilde{M}_5\tilde{M}_5^*=\tilde{M}_5^*\tilde{M}_5=
    \begin{bmatrix}
      5 & 2 & - & - & -\\
      2 & 5 & - & - & -\\
      - & - & 5 & 2 & -\\
      - & - & 2 & 5 & -\\
      - & - & - & - & 5
    \end{bmatrix}.
  \end{equation}
  This matrix $\tilde{M}_5$ is maximal determinant (see Table \ref{tab-Search5}); its determinant has absolute value $\sqrt{1701}$. The $11\times 11$ matrix
  \begin{equation}
    \tilde{M}_{11} = \left[\begin{array}{*{11}{c}}
        0 & 1 & 0 & 0 & 1 & 1 & 2 & 1 & 1 & 1 & 1\\
        2 & 0 & 0 & 0 & 1 & 2 & 1 & 2 & 2 & 1 & 1\\
        0 & 0 & 2 & 0 & 2 & 2 & 1 & 1 & 1 & 2 & 1\\
        0 & 0 & 0 & 2 & 2 & 1 & 1 & 2 & 1 & 1 & 2\\
        1 & 1 & 2 & 2 & 2 & 0 & 0 & 0 & 2 & 1 & 1\\
        1 & 2 & 1 & 1 & 0 & 0 & 1 & 0 & 1 & 1 & 1\\
        2 & 1 & 2 & 1 & 0 & 2 & 0 & 0 & 1 & 1 & 2\\
        2 & 1 & 1 & 2 & 0 & 0 & 0 & 2 & 1 & 2 & 1\\
        1 & 1 & 2 & 1 & 1 & 1 & 1 & 2 & 0 & 2 & 0\\
        1 & 1 & 1 & 2 & 1 & 2 & 1 & 1 & 0 & 0 & 2\\
        2 & 1 & 1 & 1 & 2 & 1 & 1 & 1 & 2 & 0 & 0\\
      \end{array}\right],
  \end{equation}
  satisfying
  \begin{equation}
    \tilde{M}_{11}\tilde{M}_{11}^* =
    \left[
      \begin{array}{*{11}{c}}
        11 & 2 & 2 & 2 & - & - & - & - & - & - & -\\
        2  &11 & 2 & 2 & - & - & - & - & - & - & -\\
        2  & 2 &11 & 2 & - & - & - & - & - & - & -\\
        2  & 2 & 2 &11 & - & - & - & - & - & - & -\\
        -  & - & - & - &11 & 2 & 2 & 2 & - & - & -\\
        -  & - & - & - & 2 &11 & 2 & 2 & - & - & -\\
        -  & - & - & - & 2 & 2 &11 & 2 & - & - & -\\
        -  & - & - & - & 2 & 2 & 2 &11 & - & - & -\\
        -  & - & - & - & - & - & - & - &11 & 2 & 2\\
        -  & - & - & - & - & - & - & - & 2 &11 & 2\\
        -  & - & - & - & - & - & - & - & 2 & 2 &11
      \end{array}
      \right],
  \end{equation}
  is maximal determinant; see Table \ref{tab-Search11}.
\end{example}

\begin{table}
  \caption{Tables of runtime statistics for the certificates of maximality of candidate matrices $\tilde{M}_5$, $\tilde{M}_8$, and $\tilde{M}_{11}$ of orders $n=5,8,$ and $11$ respectively. We remark that the intermediate values of $|\mathcal{G}_r|$ and $|\Phi_n^{(r)}|$ for $1\leq r<n$ are not absolute and depend on the choice of representatives of the equivalence classes of matrices in $\mathcal{G}_r$. Because of this, a parallelized algorithm may give randomized results for intermediate values. Nevertheless, the sets $\mathcal{G}_n$ and $\Phi_n^{(n)}$ do not depend on the choice of representatives during runtime.}
  \centering
  \begin{subtable}{0.8\textwidth}
    \centering
    \caption{Number of elements in $\mathcal{G}_r$ and $\Phi_5^{(r)}$ at every step of the computation of a maximality certificate for $\tilde{M}_5$.}
    \label{tab-Search5}
    \begin{tabular}{l|ccccc}
      $r$ & 1 & 2 & 3 & 4 & 5\\
      \hline
      $|\mathcal{G}_r|$ & 1 & 2 & 3 & 2 & 1\\
      $|\Phi_5^{(r)}|$  & 6 & 2 & 2 & 2 & 2
    \end{tabular}
  \end{subtable}
  \begin{subtable}{0.8\textwidth}
    \centering
    \caption{Number of elements in $\mathcal{G}_r$ and $\Phi_8^{(r)}$ at every step of the computation of a maximality certificate for $\tilde{M}_8$.}\label{tab-Search8}
    \begin{tabular}{l|cccccccc}
      $r$ & 1 & 2 & 3 & 4 & 5 & 6 & 7 & 8\\
      \hline
      $|\mathcal{G}_r|$ & 1 & 3 & 9 & 17 & 13 & 8 & 5 & 4\\
      $|\Phi_8^{(r)}|$& 14 & 4 & 4 & 4 & 2 & 2 & 2 & 2
    \end{tabular}
  \end{subtable}
  \begin{subtable}{0.8\textwidth}
    \centering
    \caption{Number of elements in $\mathcal{G}_r$ and $\Phi_8^{(r)}$ at every step of the computation of a maximality certificate for $\tilde{M}_8$.}\label{tab-Search11}
    \begin{tabular}{l|ccccccccccc}
      $r$ & 1 & 2 & 3 & 4 & 5 & 6 & 7 & 8 & 9 & 10 & 11\\
      \hline
      $|\mathcal{G}_r|$ & 1  & 3 & 16 & 152 & 718 & 458 & 109 & 24 & 8 & 1 & 1\\ 
      $|\Phi_{11}^{(r)}|$  & 25 & 4 & 4  & 4   & 4   & 4   & 4   & 2  & 2 & 2 & 2
    \end{tabular}
  \end{subtable}
\end{table}

Large determinant matrices over the third roots of order $n=14,15,16,17,19,$ and $20$ can be found in Chapter 5 and Appendix B of \cite{Ponasso-Thesis}. These matrices give the determinant values stated in Table \ref{tab-Maxdet3}.

\section{Maximal determinants of matrices over the fourth roots}\label{sec-4Roots}
We conclude with an account of results for matrices over the fourth roots; this case was first studied by J.H.E. Cohn \cite{Cohn-ComplexDOptimal}. This maximal determinant problem is much better behaved that one over the third roots mainly because of two reasons 1) we have an inclusion $\mu_2\subset \mu_4$, which implies that all Hadamard and Barba matrices with entries $\pm 1$ are also maximal determinant over the fourth roots; and 2) there is a mapping from $\mu_4=\{\pm 1,\pm i\}$ into $2\times 2$ matrices over $\mu_2$, called the \textit{Turyn morphism}
\begin{equation}
  \pm 1\mapsto \pm
  \begin{bmatrix}
    1 & -\\
    1 & 1
  \end{bmatrix};
  ~
  \pm i\mapsto \pm \begin{bmatrix}
    1 & 1\\
    - & 1
    \end{bmatrix}
\end{equation}
which sends $\BH(n,4)$ matrices into $\BH(2n,2)$ matrices, when it is applied entrywise. This morphism not only maps Hadamard matrices into Hadamard matrices of twice the size; more generally it relates the maximal determinant problem over the fourth roots to the maximal determinant problem over $\pm 1$ for matrices of twice the size.

\begin{definition} A matrix $M$ of order $2n$ is a \text{skew block matrix} if and only if $M$ has the following block-structure
  \begin{equation}
    M=\begin{bmatrix}
    A & B\\
    -B & A
    \end{bmatrix},
  \end{equation}
  where $A$ and $B$ are of order $n$.
\end{definition}

The following theorem can be obtained by analyzing the value of the determinant of a $\pm 1$ matrix obtained by applying the Turyn morphism to a matrix over the fourth roots.
\begin{theorem}[Cohn; Theorem 1 \cite{Cohn-ComplexDOptimal}]\label{thm-CohnTheorem} Let $\gamma_m(n)$ be the maximum absolute value of a matrix of order $n$ over the $\ell$-th roots. Then, $\gamma_2(2n)\geq 2^n\gamma_4(n)^2$. Furthermore, equality holds if and only if there exists a skew $\pm 1$ block matrix $M$ with $|\det(M)|=\gamma_2(2n)$.
\end{theorem}
\begin{theorem}[Cohn; Theorem 3 \cite{Cohn-ComplexDOptimal}]\label{thm-Cohn3} Suppose that $2^n\gamma_4(n)^2=\gamma_2(2n)$, then there exists a Barba matrix of order $n$ over the fourth roots.
\end{theorem}

There is a construction due to Koukovinos, Kounias, and Seberry  \cite{KKS-EW}, for maximal determinant matrices with entries $\pm 1$ at orders $2(q^2+q+1)\equiv 2\pmod{4}$ with the block structure
\begin{equation}
  \begin{bmatrix}
  R & S\\
  -S^{\intercal} & R^{\intercal}
  \end{bmatrix},
\end{equation}
where $R$ and $S$ are circulant matrices of order $q^2+q+1$. In \cite{Cohn-NumberDOptimal} it is shown that if $R$ and $S$ are circulant matrices of order $n$, then letting $P$ be the back-diagonal matrix of order $n$
\begin{equation}\label{eq-SkewEquation}
  \begin{bmatrix}
    P & 0\\
    0 & I_{n}
  \end{bmatrix}
  \begin{bmatrix}
    R & S\\
    -S^{\intercal} & R^{\intercal}
  \end{bmatrix}
  \begin{bmatrix}
    P & 0\\
    0 & I_{n}
  \end{bmatrix}
  =\begin{bmatrix}
  PRP & PS\\
  -PS^{\intercal} & R^{\intercal}
  \end{bmatrix},
\end{equation}
is a block skew matrix.

\begin{theorem} Let $q$ be a prime power. Then, there exists a Barba matrix over the fourth roots of order $q^2+q+1$.
\end{theorem}
\begin{proof} The matrix $P$ in Equation \ref{eq-SkewEquation} is a permutation matrix. Hence, for every prime power $q$ there is a $\pm 1$ skew block matrix $M$ of order $2(q^2+q+1)$ with $|\det(M)|=\gamma_2(2n)$. The result now follows from Theorem \ref{thm-CohnTheorem} and Theorem \ref{thm-Cohn3}.
\end{proof}

\begin{table}  
  \centering
  \caption{Maximal determinants and record determinants for matrices over the fourth roots. The columns labeled $R$ contain the ratio of the largest known determinant to the Hadamard bound for $n$ even, and to the Barba bound for $n$ odd. The columns labeled $|\det|^2/2^{n-1}$ contain the square of the absolute value of the largest known determinant divided by $2^{n-1}$. The symbol \red{??} is used to indicate that there is not yet a proof of maximality known for the corresponding matrix.}\label{tab-Maxdet4}
\begin{tabular}{|ccc||ccc||ccc||ccc|}
\hline
$n$ & $|\det|^2$ & R & $n$ & $|\det|^2/2^{n-1}$ & R & $n$ & $|\det|^2$ & R & $n$ & $|\det|^2/2^{n-1}$ & R\\
\hline
&&&1 & $1$ & $1$ & 2 & $2^2$ & $1$ & 3 & $ 5$ & $1$\\
4 & $4^4$ & $1$ & 5 & $2^4\times 3^2$ & $1$ & 6 & $6^6$ & $1$ & 7 & $3^6\times 13$ & $1$\\
8 & $8^8$ & $1$ & 9 & $4^{8}\times 17$ & $1$ & 10 & $10^{10}$ & $1$ & 11 & $2^2\times 5^{11} $\textcolor{red}{??} & \textcolor{red}{$0.97$}\\ 
12& $12^{12}$ & $1$ &13 & $ 6^{12}\times 5^2$ & $1$ & 14 & $14^{14}$ & $1$ & 
15 & $7^{14}\times 29$ & $1$\\
16 & $16^{16}$ & $1$ & 
17 &$13\times 137^4\times 1327^2$\textcolor{red}{??} &\textcolor{red}{$0.93$} &
18 & $18^{18}$ & $1$ &
 19 &
$3^{36}\times 37$ 
  & $1$ \\
20 & $20^{20}$ & $1$ & 21 & $10^{20}\times 41$ & $1$ & 22 & $22^{22}$ & $1$ & 23 &$3^2 \times 5\times 11^{22}$ &$1$\\
24 & $24^{24}$ & $1$ & 
$25$ & $2^{48}\times 3^{24}\times 7^2$ & $1$ &
26 & $26^{26}$ & $1$ & $27$ & $13^{26} \times 53$ & $1$\\
\hline
\end{tabular}
\end{table}

It is conjectured that there is a $\BH(n,4)$ matrix for all even $n$. From \cite{Djokovic-GoodMatrices}, \cite{Szollosi-Thesis}, and \cite{Ponasso-Thesis}, we believe that the smallest open case for the existence of a $\BH(n,4)$ matrix is $n=94$. The existence of a Barba matrix of order $n$ over the fourth roots implies that $2n-1$ is a sum of two squares (Theorem 2; \cite{Cohn-ComplexDOptimal}). The smallest open case for the maximal determinant problem over the fourth roots is $n=11$. The largest determinant we know is $2^{12}\cdot 5^{11}$; achieved by the matrix (written logarithmically):
\begin{equation}\label{eq-MaxDetCandidate11-4}
    W_{11}=W_{11}^{\intercal}:=\left[
      \begin{array}{*{11}{c}}
        3 & 0 & 1 & 0 & 1 & 0 & 1 & 3 & 2 & 2 & 3\\
        0 & 0 & 0 & 3 & 1 & 2 & 3 & 0 & 2 & 3 & 2\\
        1 & 0 & 0 & 1 & 2 & 2 & 2 & 2 & 1 & 2 & 3\\
        0 & 3 & 1 & 0 & 0 & 2 & 3 & 2 & 3 & 2 & 0\\
        1 & 1 & 2 & 0 & 0 & 2 & 2 & 3 & 2 & 1 & 2\\
        0 & 2 & 2 & 2 & 2 & 0 & 0 & 2 & 2 & 2 & 2\\
        1 & 3 & 2 & 3 & 2 & 0 & 3 & 0 & 1 & 1 & 0\\
        3 & 0 & 2 & 2 & 3 & 2 & 0 & 0 & 0 & 1 & 3\\
        2 & 2 & 1 & 3 & 2 & 2 & 1 & 0 & 0 & 2 & 1\\
        2 & 3 & 2 & 2 & 1 & 2 & 1 & 1 & 2 & 0 & 0\\
        3 & 2 & 3 & 0 & 2 & 2 & 0 & 3 & 1 & 0 & 0
      \end{array} 
    \right],
\end{equation}
satisfying:
\begin{equation}
  W_{11}W_{11}^* = W_{11}^*W_{11}=
  \left[
    \begin{array}{*{11}{c}}
      11 & 1 & 1 & 1 & 1 & 1 & - & - & - & - & -\\
      1 & 11 & 1 & 1 & 1 & - & - & - & - & - & -\\
      1 & 1 & 11 & 1 & 1 & - & - & - & - & - & -\\
      1 & 1 & 1 & 11 & 1 & - & - & - & - & - & -\\
      1 & 1 & 1 & 1 & 11 & - & - & - & - & - & -\\
      1 & - & - & - & - & 11 & 1 & - & - & - & -\\
      - & - & - & - & - & 1 & 11 & 1 & 1 & 1 & 1\\
      - & - & - & - & - & - & 1 & 11 & 1 & 1 & 1\\
      - & - & - & - & - & - & 1 & 1 & 11 & 1 & 1\\
      - & - & - & - & - & - & 1 & 1 & 1 & 11 & 1\\
      - & - & - & - & - & - & 1 & 1 & 1 & 1 & 11
    \end{array}
  \right].
\end{equation}

\section*{Acknowledgments}
The matrix $W_{11}$ in Equation \ref{eq-MaxDetCandidate11-4} was communicated by Adam Szolt Wagner. While carrying this work, the author was supported by the Japanese Society for the Promotion of Science (JSPS) as a JSPS post-doctoral fellow and through the JSPS KAKENHI Grant Number 24KF0176.
\printbibliography

@book{Horn-Johnson,
  author={Horn, Roger A. and Johnson, Charles R.},
  title={Matrix Analysis},
  publisher={Princeton University Press, Princeton NJ},
  year={2007},
  edition={2}
}

@article{Chan-TypeII,
author={Chan, Ada
and Godsil, Chris},
title={Type-II matrices and combinatorial structures},
journal={Combinatorica},
year={2010},
volume={30},
number={1},
pages={1--24},
doi={10.1007/s00493-010-2329-1}
}

@article{Seberry-Applications,
author={Seberry, Jennifer
and JWysocki, Beata
and AWysocki, Tadeusz},
title={On some applications of Hadamard matrices},
journal={Metrika},
year={2005},
month={Nov},
day={01},
volume={62},
number={2},
pages={221--239},
doi={10.1007/s00184-005-0415-y}
}

@article{Assmus-Key,
  title={Hadamard matrices and their designs: a coding-theoretic approach},
  author={Edward F. Assmus and Jennifer D. Key},
  journal={Trans. Am. Math. Soc.},
  year={1992},
  volume={330},
  pages={269--293}
}

@phdthesis{Zauner,
  title={Quantendesigns. Grundzüge einer nichtkommutativen Designtheorie},
  author={Zauner, G.},
  institution={Wien},
  year={1999}
}

@article{Werner-Teleportation,
doi = {10.1088/0305-4470/34/35/332},
year = {2001},
volume = {34},
number = {35},
pages = {7081},
author = {R F Werner},
title = {All teleportation and dense coding schemes},
journal = {Journal of Physics A: Mathematical and General}
}

@article{Tadej2006,
author={Tadej, Wojciech
and {\.{Z}}yczkowski, Karol},
title={A Concise Guide to Complex Hadamard Matrices},
journal={Open Systems {\&} Information Dynamics},
year={2006},
volume={13},
number={2},
pages={133--177},
doi={10.1007/s11080-006-8220-2}
}

@article{Box-Draper-71,
author = {Box, M. J.  and Draper, N. R.},
title = {Factorial Designs, the $|X'X|$ Criterion, and Some Related Matters},
journal = {Technometrics},
volume = {13},
number = {4},
pages = {731--742},
year = {1971},
doi = {10.1080/00401706.1971.10488845}
}

@article{Mitchell-DOptimal,
author = {Mitchell, Toby J.},
title = {An Algorithm for the Construction of “D-Optimal” Experimental Designs},
journal = {Technometrics},
volume = {42},
number = {1},
pages = {48--54},
year = {2000},
doi = {10.1080/00401706.2000.10485978}
}

@article{Aguiar-DOptimalTutorial,
title = {D-optimal designs},
journal = {Chemom. Intell. Lab. Syst.},
volume = {30},
number = {2},
pages = {199-210},
year = {1995},
doi = {https://doi.org/10.1016/0169-7439(94)00076-X},
author = {P.F. {de Aguiar} and B. Bourguignon and M.S. Khots and D.L. Massart and R. Phan-Than-Luu},
}

@article{Smith-DCriterion,
 author = {Smith, Kirstine},
 journal = {Biometrika},
 number = {1/2},
 pages = {1--85},
 title = {On the Standard Deviations of Adjusted and Interpolated Values of an Observed Polynomial Function and its Constants and the Guidance they give Towards a Proper Choice of the Distribution of Observations},
 volume = {12},
 year = {1918}
}

@book{Bengtsson-GQS,
  author={Bengtsson, Ingemar and Zyczkowski, Karol},
  title={Geometry of Quantum States},
  year={2006},
  publisher={Cambridge University Press}
}

@article{deLauney-SurveyGHM,
  author={de Launey, Warwick},
  title={A survey of generalised Hadamard matrices and difference matrices with large $r$},
  journal={Utilitas Math.},
  volume={30},
  year={1986},
  pages={5--29}
}

@article{CCDL,
author={Compton, B.
and Craigen, R.
and de Launey, W.},
title={Unreal $BH(n,6)$'s and Hadamard matrices},
journal={Designs, Codes and Cryptography},
year={2016},
month={May},
day={01},
volume={79},
number={2},
pages={219-229},
doi={10.1007/s10623-015-0045-y}
}

@ARTICLE{OCathain-Compressed,
  author={Bryant, Darryn and Colbourn, Charles J. and Horsley, Daniel and Ó Catháin, Padraig},
  journal={IEEE Trans. Inf. Theor.}, 
  title={Compressed Sensing With Combinatorial Designs: Theory and Simulations}, 
  year={2017},
  volume={63},
  number={8},
  pages={4850-4859},
  doi={10.1109/TIT.2017.2717584}}

@article{Vaz-Compressed,
author = {Pedro G. Vaz and Daniela Amaral and L. F. Requicha Ferreira and Miguel Morgado and Jo\~{a}o Cardoso},
journal = {Opt. Express},
number = {8},
pages = {11666--11681},
publisher = {Optica Publishing Group},
title = {Image quality of compressive single-pixel imaging using different Hadamard orderings},
volume = {28},
month = {Apr},
year = {2020},
doi = {10.1364/OE.387612}
}

@article {Gonzalez-Advances,
	author = {Gonz{\'a}lez, Ernesto {\'A}lvarez and Bal{\'a}m-Narv{\'a}ez, Ricardo and Angulo, Diego F. and Duchen, Pablo},
	title = {Advances and applications of the closest-tree algorithm and Hadamard conjugation in phylogenetic inference},
	elocation-id = {2024.12.06.627223},
	year = {2024},
	doi = {10.1101/2024.12.06.627223},
	publisher = {Cold Spring Harbor Laboratory},
	eprint = {https://www.biorxiv.org/content/early/2024/12/10/2024.12.06.627223.full.pdf},
	journal = {bioRxiv}
}

@article{Penttila-Conjugation,
  title={Remarks on Hadamard conjugation and combinatorial phylogenetics},
  author={McBee, C.D. and Penttila, T.},
  journal={Australas. J. Comb.},
  year={2016},
  volume={66},
  number={2},
  pages={177--191}
}

@article{Hendy-Trees,
  title={A framework for the quantitative study of evolutionary trees},
  author={Hendy, Michael D. and Penny D.},
  journal={Syst. Zool.},
  volume={38},
  number={4},
  pages={297--309},
  year={1989}
}

@article{Hendy-Phyl,
  title={Spectral Analysis of Phylogenetic Data},
  author={Hendy, Michael D.},
  journal={J. Classif.},
  volume={10},
  pages={5--24},
  year={1993}
}

@inproceedings{Turyn,
  title={Complex Hadamard Matrices},
  author={Turyn, R. J.},
  booktitle={Proc. Calgary Internat. Conf., Calgary, Alta., 1969, Gordon and Breach, New York},
  year={1970},
  pages={435--437}
}

@article{OCathain-Morphisms,
title = {Morphisms of Butson classes},
journal = {Linear Algebra and its Applications},
volume = {577},
pages = {78-93},
year = {2019},
issn = {0024-3795},
doi = {https://doi.org/10.1016/j.laa.2019.04.020},
author = {Ronan Egan and Padraig {Ó Catháin}}
}

@article{Paley,
  title={On Orthogonal Matrices},
  author={Paley, R. E. A. C.},
  journal={J. Math. Phys},
  volume={12},
  pages={311--320},
  year={1933},
  doi={10.1002/sapm1933121311}
}

@report{Brouwer,
author = {Brouwer, A. E.},
title  = {An infinite series of symmetric designs},
note   = {Report ZW202/83},
institution = {Math. Zentrum, Amsterdam},
year   = {1983}
}

@article{Neubauer-Radcliffe,
title = {The maximum determinant of ± 1 matrices},
journal = {Linear Algebra and its Applications},
volume = {257},
pages = {289-306},
year = {1997},
issn = {0024-3795},
doi = {https://doi.org/10.1016/S0024-3795(96)00147-4},
author = {M.G. Neubauer and A.J. Radcliffe}
}

@article{Djokovic-GoodMatrices,
  journal={J. Comb. Math. Comb. Comput.},
  title={Good Matrices of Orders 33, 35 and 127 Exist},
  author={Đokovic, Dragomir Ž.},
  year={1993},
  volume={14},
  pages={145--152},
  sortname={Djokovic, Dragomir}
}

@article{KKS-EW,
title = {Supplementary difference sets and optimal designs},
journal = {Discrete Math.},
volume = {88},
pages = {49-58},
year = {1991},
doi = {https://doi.org/10.1016/0012-365X(91)90058-A},
author = {Christos Koukouvinos and Stratis Kounias and Jennifer Seberry}
}

@article {MK-21,
    AUTHOR = {Chadjipantelis, Theo and Kounias, Stratis and Moyssiadis,
              Chronis},
     TITLE = {The maximum determinant of {$21\times 21$}
              {$(+1,-1)$}-matrices and {$D$}-optimal designs},
   JOURNAL = {J. Statist. Plann. Inference},
  FJOURNAL = {Journal of Statistical Planning and Inference},
    VOLUME = {16},
      YEAR = {1987},
    NUMBER = {2},
     PAGES = {167--178},
MRREVIEWER = {Ivar\ Petersen},
       DOI = {10.1016/0378-3758(87)90066-8}
}

@phdthesis{deLauney-Thesis,
	author = {De Launey, Warwick},
	title = {O, G designs and applications},
	url = {https://hdl.handle.net/2123/31928},
	pages = {},
	year = {1987},
	school = {University of Sydney, NSW, Australia}
}

@article {Butson,
    AUTHOR = {Butson, Alton T.},
     TITLE = {Generalized {H}adamard matrices},
   JOURNAL = {Proc. Amer. Math. Soc.},
  FJOURNAL = {Proceedings of the American Mathematical Society},
    VOLUME = {13},
      YEAR = {1962},
     PAGES = {894--898},
      ISSN = {0002-9939},
   MRCLASS = {15.45},
  MRNUMBER = {142557},
MRREVIEWER = {H. Gupta}
}

@article {Winterhof-NonexistenceButson,
    AUTHOR = {Winterhof, Arne},
     TITLE = {On the non-existence of generalized {H}adamard matrices},
   JOURNAL = {J. Statist. Plann. Inference},
  FJOURNAL = {Journal of Statistical Planning and Inference},
    VOLUME = {84},
      YEAR = {2000},
    NUMBER = {1-2},
     PAGES = {337--342},
      ISSN = {0378-3758},
   MRCLASS = {05B20},
  MRNUMBER = {1747512}
}

@article{Scarpis,
author = {Scarpis, Umberto},
title = {Sui Determinanti di Valore Massimo},
journal = {Rendiconti del Reale Istituto Lombardo di Scienze e Lettere},
volume ={31},
pages ={1441–1446},
year={1898}
}

@book {Marcus,
    AUTHOR = {Marcus, Daniel A.},
     TITLE = {Number fields},
    SERIES = {Universitext},
    EDITION= {2},
 PUBLISHER = {Springer, Cham},
      YEAR = {2018},
  MRNUMBER = {3822326}
       %DOI = {10.1007/978-3-319-90233-3},
}

@book {Storer-CyclotomyBook,
    AUTHOR = {Storer, Thomas},
     TITLE = {Cyclotomy and difference sets},
    SERIES = {Lectures in Advanced Mathematics},
    VOLUME = {No. 2},
 PUBLISHER = {Markham Publishing Co., Chicago, IL},
      YEAR = {1967}
}

@article {Orrick-15,
    AUTHOR = {Orrick, William P.},
     TITLE = {The maximal {$\{-1,1\}$}-determinant of order 15},
   JOURNAL = {Metrika},
  FJOURNAL = {Metrika. International Journal for Theoretical and Applied
              Statistics},
    VOLUME = {62},
      YEAR = {2005},
    NUMBER = {2-3},
     PAGES = {195--219},
       DOI = {10.1007/s00184-005-0410-3}
}

@article{Munemasa-Cyclotomic,
    author = {Bannai, Eichi and Munemasa, Akihiro},
    title = {Davenport-{H}asse theorem and cyclotomic association schemes},
    journal = {Proc. Algebr. Comb. Hirosaki Univ.},
    year = {1990}
}

@article {Moyssiadis-Kounias,
    AUTHOR = {Moyssiadis, Chronis and Kounias, Stratis},
     TITLE = {The exact {$D$}-optimal first order saturated design with
              {$17$} observations},
   JOURNAL = {J. Statist. Plann. Inference},
  FJOURNAL = {Journal of Statistical Planning and Inference},
    VOLUME = {7},
      YEAR = {1983},
     PAGES = {13--27},
}

@article{Kharaghani-BushType,
	author={Kharaghani, Hadi},
	title={New classes of weighing matrices},
	journal={Ars Comb.},
	volume={9},
	year={1985},
	pages={69--72}
}

@book {Hall-CombinatorialTheoryBook,
    AUTHOR = {Hall, Jr., Marshall},
     TITLE = {Combinatorial theory},
    SERIES = {Wiley Classics Library},
   EDITION = {second},
      NOTE = {A Wiley-Interscience Publication},
 PUBLISHER = {John Wiley \& Sons, Inc., New York},
      YEAR = {1998}
}

@book {Bannai-Ito,
    AUTHOR = {Bannai, Eiichi and Ito, Tatsuro},
     TITLE = {Algebraic combinatorics. {I}},
      NOTE = {Association schemes},
 PUBLISHER = {The Benjamin/Cummings Publishing Co., Inc., Menlo Park, CA},
      YEAR = {1984}
}

@misc{Babai-FourierAbelianGroups,
    author={Babai, László},
    title={{The Fourier Transform and
Equations over Finite Abelian Groups}},
    subtitle={An introduction to the method of
trigonometric sums},
    publisher={University of Chicago},
    url={https://people.cs.uchicago.edu/~laci/HANDOUTS/fourier.pdf},
    note={version 1.4},
    year={2023}
}

@article{OCathain-SurveyMaxDet,
    author = {Browne, Patrick and Egan, Ronan and Hegarty, Fintan and \'O{} Cath\'ain, Padraig},
    title = {A survey of the {H}adamard maximal determinant problem},
    journal = {Electron. J. Comb.},
    year = {2021},
    volume={28}
}

@article{Wojtas-Determinants,
    author={Wojtas, Mieczys\l{}aw},
    title={{On Hadamard's inequality for the determinants of order non-divisible by {$4$}}},
    journal={Colloq. Math.},
    volume={12},
    year={1964},
    pages={73--83}
}

@article{Barba-DetBound,
    author = {Barba, Guido},
    title = {Intorno al teorema di Hadamard sui determinanti a valore massimo},
    journal = {Giorn. Mat. Battaglini},
    year = {1933},
    volume={71},
    pages={70--86}
}

@article{Cohn-ComplexDOptimal,
    author = {Cohn, John H.E.},
    title = {Complex $D$-optimal designs},
    journal = {Discrete Math.},
    year = {1996},
    volume={156},
    pages={237--241},
    doi={10.1016/0012-365X(96)00035-0}
}

@article{Cohn-NumberDOptimal,
title = {On the number of D-optimal designs},
journal = {J. Comb. Theory Ser. A},
volume = {66},
number = {2},
pages = {214-225},
year = {1994},
issn = {0097-3165},
doi = {https://doi.org/10.1016/0097-3165(94)90063-9},
author = {J.H.E Cohn}
}

@phdthesis{Ponasso-Thesis,
      title={{Combinatorics of Complex Maximal Determinant Matrices}}, 
      author={Guillermo Nuñez Ponasso},
      month={8},
      year={2023},
      school={Worcester Polytechnic Institute.},
      eprint={2404.09040},
      archivePrefix={arXiv}
}

@article{Hadamard-Determinants,
    author  = {Hadamard, Jacques S.},
    title   = {{Résolution d’une question relative aux determinants}},
    journal = {Bull. des Sci. Math.},
    year    = {1893},
    volume  = {17},
    pages   = {240--246}
}

@phdthesis{Szollosi-Thesis,
    author = {Szöll\H{o}si, Ferenc},
    title = {{Construction, classification and parametrization of complex Hadamard matrices}},
    school = {Central European University},
    year = {2011},
    archivePrefix = {arXiv},
    eprint = {1110.5590}
}

@book{Horadam-HadamardBook,
    author = {Horadam, Kathy J.},
    title  = {Hadamard matrices and their applications},
    publisher = {Princeton University Press},
    year   = {2012}
}

@article{Sylvester-InverseOrthogonal,
  title={{Thoughts on inverse orthogonal matrices, simultaneous sign successions, and tessellated pavements in two or more colours, with applications to Newton's rule, ornamental tile-work, and the theory of numbers}},
  author={Sylvester, James J.},
  year={1867},
  journal={Lond. Edinb. Dubl. Philos. Mag. J. Sci.},
  volume={34},
  pages ={461--475}
}

@article{McKay-GraphIsoII,
title = {Practical graph isomorphism, II},
journal = {J. Symb. Comput.},
volume = {60},
pages = {94-112},
year = {2014},
doi = {https://doi.org/10.1016/j.jsc.2013.09.003},
author = {Brendan D. McKay and Adolfo Piperno}
}
\end{document}